\documentclass{amsart}

\usepackage[margin=1.35in]{geometry}
\usepackage{amsmath}
\usepackage{amssymb}
\usepackage[colorlinks=true,linkcolor=blue,citecolor=blue,urlcolor=blue]{hyperref}

\newtheorem{theorem}{Theorem}[section]
\newtheorem{prop}[theorem]{Proposition}
\newtheorem{lem}[theorem]{Lemma}
\newtheorem{cor}[theorem]{Corollary}

\theoremstyle{definition}
\newtheorem{defn}[theorem]{Definition}
\newtheorem{ass}[theorem]{Assumption}

\theoremstyle{remark}

\newtheorem{remark}[theorem]{Remark}

\numberwithin{equation}{section}

\newcommand{\be}{\begin{equation}}
\newcommand{\ee}{\end{equation}}

\newcommand{\lb}{\left(}
\newcommand{\rb}{\right)}
\newcommand{\lsb}{\left[}
\newcommand{\rsb}{\right]}
\newcommand{\lcb}{\left\{}
\newcommand{\rcb}{\right\}}

\newcommand{\bbr}{\bar{r}}

\newcommand{\ip}[1]{\langle#1\rangle}
\newcommand{\norm}[1]{\lVert#1\rVert}
\newcommand{\abs}[1]{\left|#1\right|}

\newcommand{\wt}{\widetilde}
\newcommand{\wh}{\widehat}

\newcommand{\ve}{\varepsilon}

\newcommand{\x}{f}
\newcommand{\y}{g}
\newcommand{\z}{h}

\newcommand{\A}{\xi}
\newcommand{\B}{\mathcal{B}}
\newcommand{\C}{\mathcal{C}}
\newcommand{\E}{\mathbb{E}}
\newcommand{\F}{\mathcal{F}}
\renewcommand{\L}{L}
\newcommand{\M}{\mathbb{M}}
\newcommand{\N}{\mathbb{N}}
\renewcommand{\P}{\mathbb{P}}

\newcommand{\R}{\mathbb{R}}

\newcommand{\U}{\mathcal{N}}
\newcommand{\X}{X}
\newcommand{\Y}{Y}
\newcommand{\Z}{Z}

\newcommand{\bm}{W}

\newcommand{\phib}{\mathcal{J}}
\newcommand{\etab}{\mathcal{K}}
\newcommand{\psib}{\mathcal{H}}

\newcommand{\Xib}{{\Xi}}

\newcommand{\lyap}{M}

\newcommand{\allN}{\mathcal{I}}

\newcommand{\lip}{\kappa}

\newcommand{\spaan}{\text{span}}
\newcommand{\conv}{{\text{cone}}}
\newcommand{\dom}{\text{dom}}

\newcommand{\cts}{\mathbb{C}}
\newcommand{\dlr}{\mathbb{D}_{\text{l,r}}}
\newcommand{\dr}{\mathbb{D}}

\newcommand{\state}{\mathbb{X}}

\newcommand{\sphere}{\mathbb{S}^{J-1}}

\newcommand{\hyper}{\mathbb{H}}

\newcommand{\sm}{{\Gamma}}
\newcommand{\dm}{\Lambda}

\newcommand{\proj}{\mathcal{L}}

\begin{document}

\title[Sensitivity analysis for RBM]{Sensitivity analysis for the stationary distribution of reflected Brownian motion in a convex polyhedral cone*}

\date{\today}

\author[David Lipshutz]{David Lipshutz}
\address{Faculty of Electrical Engineering \\ Technion --- Israel Institute of Technology \\ Haifa, Israel}
\email{lipshutz@technion.ac.il}

\author[Kavita Ramanan]{Kavita Ramanan}
\address{Division of Applied Mathematics \\ 
                Brown University \\ 
                Providence, Rhode Island, USA}
\email{kavita{\_}ramanan@brown.edu}



\thanks{*Research of the second author was supported in part by NSF grant DMS-1713032 and a Simons Fellowship. The first author was also supported in part at the Technion by a Zuckerman fellowship.}

\dedicatory{Technion --- Israel Institute of Technology and Brown University}

\begin{abstract}
Reflected Brownian motion (RBM) in a convex polyhedral cone arises in a variety of applications ranging from the theory of stochastic networks to mathematical finance, and under general stability conditions, it has a unique stationary distribution.  
In such applications, to implement a stochastic optimization algorithm or quantify robustness of a model, it is useful to characterize the dependence of stationary  performance measures on model parameters. 
In this work we characterize parametric sensitivities of the stationary distribution of an RBM in a simple convex polyhedral cone; that is, sensitivities to perturbations of the parameters that define the RBM --- namely, the covariance matrix, drift vector and directions of reflection along the boundary of the polyhedral cone. 
In order to characterize these sensitivities we study the long time behavior of the joint process consisting of an RBM along with its so-called derivative process, which characterizes pathwise derivatives of RBMs on finite time intervals. 
We show that the joint process is positive recurrent, has a unique stationary distribution, and parametric sensitivities of the stationary distribution of an RBM can be expressed in terms of the stationary distribution of the joint process. 
This can be thought of as establishing an interchange of the differential operator and the limit in time.
The  analysis of  ergodicity of the joint process is significantly more complicated than that of the RBM due to its degeneracy and the fact that the derivative process exhibits jumps that are modulated by the RBM.
The proofs of our results rely on path properties of coupled RBMs and contraction properties related to the geometry of the polyhedral cone and directions of reflection along the boundary.
Our results are potentially useful for developing efficient numerical algorithms for  computing sensitivities of functionals of stationary RBMs.
\end{abstract}

\maketitle

\section{Introduction}\label{sec:intro}

\subsection{Overview}

Reflected Brownian motions (RBMs) in convex polyhedral cones (with oblique directions of reflection) arise in a variety of applications, including as heavy traffic limits in queueing theory \cite{Harrison2013,Kelly1996,Kushner2013,Peterson1991,Reiman1984} and in the study of rank-based diffusion models used in mathematical finance \cite{Banner2005,Ichiba2011}. 
The stationary distribution of such an RBM is of central importance in applications due to its use in approximating equilibrium properties such as queue lengths, market capitalizations, etc.
Under suitable stability conditions (see \cite{Budhiraja1999,Dupuis1994}), an RBM in a convex polyhedral cone has a unique stationary distribution, which is completely characterized by its covariance matrix, drift vector and the directions of reflection along the boundary of the polyhedral cone. 
For applications in uncertainty quantification and stochastic optimization, among other areas, it is of interest to characterize sensitivities of the stationary distribution with respect to the parameters that describe an RBM.

In \cite{Lipshutz2017}, pathwise differentiability over a finite time interval was established for RBMs in convex polyhedral cones with oblique directions of reflection that satisfy a geometric condition that ensures the associated deterministic Skorokhod map (SM) is Lipschitz continuous (see Assumption \ref{ass:setB}).
In addition, the pathwise derivative of an RBM was characterized in terms of a so-called derivative process, which is governed by a linear constrained stochastic differential equation with jumps whose drift and diffusion coefficients, domain and directions of reflection are modulated by the RBM (see Theorem \ref{thm:pathwise}). 
In this work, we first show that the joint process consisting of the RBM and the derivative process form a Feller continuous Markov process (we henceforth simply refer to this pair of processes as the \emph{joint process}).  
Next, under standard stability conditions that guarantee that the RBM is positive recurrent (see \cite{Budhiraja1999,Dupuis1994}),  we show that the joint process is also ergodic and positive recurrent. 
We then show that sensitivities of the stationary distribution of the RBM can be characterized in terms of the stationary distribution of the joint process.
In particular, our result implies that derivatives of expectations of certain functionals of a stationary RBM can be expressed as the expectation of an associated functional of the stationary joint process (see Theorem \ref{thm:sensitivity}). 
This is potentially useful for developing efficient numerical algorithms for computing sensitivities of such functionals.
For example, in related work \cite{Lipshutz2017a}, the joint process was shown to be useful for estimating sensitivities of functionals of RBMs on finite time intervals.

The proofs of these results entail several novel arguments.
The proof of the Feller Markov property relies on continuity properties of the so-called derivative map, a deterministic map on path space that is useful in the analysis of the derivative process, which were established in \cite{Lipshutz2017a,Lipshutz2018}. 
The proof of ergodicity of the joint process (under the stability condition on the RBM) entails a careful analysis of the derivative process, whose dynamics are unusual due to the fact that the process jumps at the singular set of times when the RBM hits the boundary of the cone.  
To show that the derivative process is stable we prove the RBM repeatedly ``visits'' every face of the cone and use contraction properties of so-called derivative projection operators, which are related to the jumps of the derivative process.
Establishing uniqueness of the stationary distribution of the joint process is also nonstandard due to the fact that the RBM and its derivative process are driven by a common Brownian motion, and hence the pair is degenerate. 
We use an asymptotic coupling method (see \cite{Hairer2011}), which also relies on the aforementioned contraction properties.
The final step is to justify an interchange of the differential operator and the limit in time, which implies that sensitivities of the stationary distribution can be expressed in terms of the stationary distribution of the joint process.

In summary, the main contributions of this work are as follows:
\begin{itemize}
	\item Feller Markov property of the joint process (Section \ref{sec:joint}); 
	\item Ergodicity and positive recurrence of the joint process (Section \ref{sec:jointstable}); 
	\item Sensitivity analysis for the stationary distribution of an RBM (Section \ref{sec:sensitivityinvariant}).
\end{itemize}{
There are many interesting directions for future work.  
Firstly, it would be natural to generalize these results to more general (state-dependendent) reflected diffusions in more general domains with more general reflection vector fields.  
Pathwise differentiability has been established for reflected diffusions in polyhedral domains with state-dependent drift and dispersion coefficients in \cite{Lipshutz2017}, but the proof of ergodicity of the joint process in this case would entail more technical challenges.
Secondly, it would be of interest to use these results to develop numerically efficient algorithms for computing sensitivities of reflected diffusions.

\subsection{Prior results}

There are relatively few results on sensitivities of the stationary distribution of an RBM in a convex polyhedral cone. 
One exception is the work by Dieker and Gao \cite{Dieker2014}, which considers a particular sub-class of reflected diffusions and perturbations, and in that setting provides a characterization of the stationary distribution of the joint process (assuming it exists) in terms of a basic adjoint relation.
However, they do not prove existence or uniqueness of a stationary distribution, they only allow for perturbations of the drift in the specific direction $-{\bf 1}$ (the vector with negative one in each component), and they restrict to a sub-class of reflected diffusions (namely those in the orthant whose reflection matrices are so-called $\mathcal{M}$-matrices, which satisfy useful monotonicity properties and are described in Lemma \ref{lem-setB}).  
While they consider reflected \textit{diffusions}, they avoid many of the complications that arise from a state-dependent drift and covariance matrix by only allowing perturbations of the drift in the direction $-{\bf 1}$.
Their work is motivated in part by their conjecture that sensitivities of the stationary distribution can be expressed in terms of the stationary distribution of the joint process; however, that is not the focus of their work and they only demonstrate that the conjecture holds in the one-dimensional setting using techniques that do not extend to the multidimensional setting.  
Our main result affirms their conjecture for RBMs in a more general multidimensional setting than they consider, where we also allow for perturbations to the covariance matrix and directions of reflection, and treat a more general class of polyhedral domains and directions of reflection. 
Sensitivities of RBMs with respect to the covariance matrix and directions of reflection are relevant in applications and reflection matrices that do not fall into the class considered in \cite{Dieker2014} arise in applications such as multiclass queueing networks (see, e.g., \cite{Peterson1991}).
An interesting future direction along the lines of \cite{Dieker2014} is to characterize the stationary distribution of the joint process (in our much more general setting) in terms of a basic adjoint relation.

We also mention the work by Kushner and Yang \cite{Kushner1992}, which focuses on Monte Carlo methods for computing sensitivities of stationary distributions of \emph{unconstrained} diffusions to perturbations of the drift. 
The numerical methods developed there can likely be adapted to the constrained setting; however, the proof of their result employs a change of measure argument to relate perturbations of the drift to perturbations of the underlying measure.
Aside from some specific cases, the change of measure argument cannot be applied to perturbations of the covariance matrix or directions of reflection, which are relevant in applications. 
Their method, which involves computing perturbations to the underlying measure, is commonly referred to as a likelihood ratio method for computing sensitivities, whereas our approach of using pathwise derivatives of the RBM is referred to as an infinitesimal perturbation analysis method.
While infinitesimal perturbation analysis is often more difficult to justify in practice, when it is applicable, infinitesimal perturbation analysis estimators of sensitivities typically have much lower variance than likelihood ratio estimators (see \cite[Chapter VII]{Asmussen2007} for further details), and also allow one to more efficiently evaluate sensitivities to multiple parameters simultaneously.

\subsection{Outline of the paper}

This paper is organized as follows. 
In Section \ref{sec:RBM} we introduce an RBM and pathwise derivatives along an RBM. 
In Section \ref{sec:main} we present our main results. 
In Section \ref{sec:sdm} we state the deterministic Skorokhod problem and derivative problem, and prove a useful contraction property for solutions to the derivative problem. 
In Section \ref{sec:Feller} we prove the joint process is a Feller continuous Markov process. 
In Section \ref{sec:rbmestimates} we prove some useful estimates for the probability an RBM visits every face in a compact time interval. 
In Section \ref{sec:stable} we prove positive recurrence and stability properties for the joint process.
In Section \ref{sec:ergodic} we prove the joint process has a unique stationary distribution and in Section \ref{sec:interchange} we show that derivatives of expectations of certain functionals of a stationary RBM can be expressed in terms of the expectation of a functional of the stationary joint process.
The Appendix contains the proof of a technical lemma.

\subsection{Notation}

We now collect some notation that shall be used throughout this work. 
We let $\N=\{1,2,\dots\}$ denote the set of positive integers. 
For $J\in\N$, let $\R^J$ denote $J$-dimensional Euclidean space and $\R_+^J$ denote the non-negative orthant. 
When $J=1$, we suppress $J$ and simply write $\R$ for $(-\infty,\infty)$. 
For $r,s\in\R$ we let $r^+=\max(r,0)$, $r^-=\max(-r,0)$ and $r\wedge s=\min(r,s)$.
For a column vector $x\in\R^J$, let $x^j$ denote the $j$th component of $x$. 
We let $\bf 0$ (resp.\ $\bf 1$) denote the vector in $\R^J$ with 0 (resp.\ 1) in each component.
We write $\ip{\cdot,\cdot}$ and $|\cdot|$ for the usual inner product and norm, respectively, on $\R^J$. 
We use $\sphere:=\{x\in\R^J:|x|=1\}$ to denote the unit sphere in $\R^J$. 
For $J,K\in\N$, let $\R^{J\times K}$ denote the set of real-valued matrices with $J$ rows and $K$ columns. 
We write $M^T$ to denote the transpose of $M$. 
We let $\norm{\cdot}$ denote the Frobenius norm on $\R^{J\times M}$.
We let $\M^{J\times J}$ denote the open set of positive-definite symmetric matrices in $\R^{J\times J}$, and let $E_J\in\M^{J\times J}$ denote the $J\times J$ identity matrix.
Given a topological space $S$, we let $\B(S)$ denote the $\sigma$-algebra of Borel subsets of $S$.

For a subset $A\subseteq\R$ let $\inf A$ and $\sup A$ denote the infimum and supremum, respectively, of $A$. 
We use the convention that the infimum and supremem of the emptyset are respectively defined to be $\infty$ and $-\infty$. 
For a subset $E\subseteq\R^J$ let $\conv(E)$ denote the convex cone generated by $E$; that is,
	$$\conv(E):=\lcb\sum_{k=1}^Kr_kx_k:K\in\N,x_k\in E,r_k\geq0\rcb,$$
with the convention that $\conv(\emptyset):=\{0\}$. 
We let $\spaan(E)$ denote the set of all possible finite linear combinations of vectors in $E$ with the convention that $\spaan(\emptyset):=\{0\}$. 
We let $E^\perp:=\{x\in\R^J:\ip{x,y}=0\;\forall\;y\in E\}$ denote the orthogonal complement of $\spaan(E)$ in $\R^J$.
Given two subsets $E_1,E_2\subset\R^J$, define their Hausdorff distance by
	\be\label{eq:Hausdorff}\rho_H(E_1,E_2):=\max\lcb\sup_{x\in E_1}\inf_{y\in E_2}|x-y|,\sup_{y\in E_2}\inf_{x\in E_1}|x-y|\rcb.\ee

Given a subset $E\subseteq\R^J$, we let $\dr(E)$ denote the set of functions from $[0,\infty)$ to $E$ that are right continuous and have finite left limits (RCLL). 
We let $\cts(E)$ denote the subset of continuous functions in $\dr(E)$. Given a subset $A\subseteq E$, we let $\cts_A(E)$ denote the set of functions $f$ in $\cts(E)$ satisfying $f(0)\in A$. 
If $0\in E$, we let $\cts_0(E)$ denote the set of function $f$ in $\cts(E)$ satisfying $f(0)=0$. 
We endow $\dr(E)$ and its subsets with the Skorokhod $J_1$-topology. For $f\in\dr(E)$ we let $f(t+):=\lim_{s\downarrow t}f(s)$ for all $t\in[0,\infty)$. 
For $s\geq0$, we define the shift operator $\Theta_s$ as follows: for every function $f:[0,\infty)\mapsto\R^J$, the shifted function $\Theta_sf:[0,\infty)\mapsto\R^J$ is defined by 
	\be\label{eq:shift}(\Theta_sf)(t):= f(s+t)-f(s)\qquad\text{for all }t\geq0.\ee
Given a function $f:[0,\infty)\mapsto\R^J$ and $t>0$, let $|f|(t)$ denote the total variation of $f$ on $[0,t]$.
 Let $\iota\in\cts_0(\R_+)$ denote the identity function defined by $\iota(t)=t$ for all $t\ge0$. 

Throughout this paper we fix a filtered probability space $(\Omega,\F,\{\F_t\},\P)$ satisfying the usual conditions (see, e.g., \cite[Chapter II, Definition 67.1]{Rogers2000}). 
We abbreviate ``almost surely'' as ``a.s.''.
We write $\E$ to denote expectation under $\P$. 
By a $J$-dimensional $\{\F_t\}$-Brownian motion $\bm$ on $(\Omega,\F,\P)$ we mean that $(\bm^1,\dots,\bm^J)$ are independent and, for each $j=1,\dots,J$, $\{\bm^j(t),\F_t,t\geq0\}$ is a continuous martingale with quadratic variation $[\bm^j]_t=t$ for $t\geq0$ that starts at the origin.

\section{Pathwise differentiability of reflected Brownian motion}\label{sec:RBM}

In this section we state our main assumptions that ensure existence of pathwise unique RBMs and under which pathwise differentiability of RBMs hold.

\subsection{A family of coupled RBMs}
\label{sec:rbm}

In this section we define the family of coupled RBMs in a fixed polyhedral cone that we consider and state assumptions guaranteeing existence and pathwise uniqueness of the associated RBMs.
We first introduce the polyhedral cone, which will remain fixed throughout this work. 
Let $G$ be a nonempty simple polyhedral cone in $\R^J$ equal to the intersection of $J$ closed half spaces in $\R^J$; that is,
	$$G:=\bigcap_{i=1,\dots,J}\lcb x\in\R^J:\ip{x,n_i}\geq 0\rcb,$$
for linearly independent unit vectors $n_i\in \sphere$, $i=1,\dots,J$. 
The term ``simple'' refers to the facts that the number of faces is equal to the dimension of the space and that the corresponding normal vectors $\{n_1,\dots,n_J\}$ are linearly independent. 
For each $i=1,\dots,J$, we let $F_i:=\{x\in\partial G:\ip{x,n_i}=0\}$ denote the $i$th face. 
For notational convenience, we let $\allN:=\{1,\dots,J\}$, and for $x\in G$, we write 
	\begin{equation}\label{eq:allNx}\allN(x):=\{i\in\allN:x\in F_i\}\end{equation} 
to denote the (possibly empty) set of indices associated with the faces that intersect at $x$.

Let $U$ be an open parameter set. 
For ease of exposition we assume $U\subset\R$ is a one-dimensional set.
For each $i\in\allN$, we fix a continuously differentiable function 
	$$d_i:U\mapsto\R^J$$
that satisfies $\ip{d_i(\alpha),n_i}=1$ for all $\alpha\in U$. 
For $i\in\allN$ and $\alpha\in U$, $d_i(\alpha)$ will denote the (constant) direction of reflection along the face $F_i$ associated with the parameter $\alpha$. 
Since the directions of reflection can always be renormalized, our assumption that $\ip{d_i(\alpha),n_i}=1$ for all $\alpha\in U$ is equivalent to the necessary condition that the reflection directions point into the domain: $\langle d_i(\alpha), n_i \rangle > 0$, and hence, is without loss of generality. 
For $\alpha\in U$ and $x\in\partial G$, we let $d(\alpha,x)$ denote the cone generated by the admissible directions of reflection at $x$; that is,
	\be\label{eq:dux}d(\alpha,x):=\conv\lb\lcb d_i(\alpha),i\in\allN(x)\rcb\rb.\ee
For convenience, we extend the definition of $d(\alpha,\cdot)$ to all of $G$ by setting $d(\alpha,x):=\{0\}$ for $x\in G^\circ$. 
Recall that $\M^{J\times J}$ denotes the open set of positive-definite symmetric matrices in $\R^{J\times J}$. 
We fix continuously differentiable functions 
\begin{align*}
	a:U\mapsto\M^{J\times J},\qquad b:U\mapsto\R^J,
\end{align*}
and denote their respective Jacobians by $a':U\mapsto\R^{J\times J}$ and $b':U\mapsto\R^J$. 
For $\alpha\in U$, $a(\alpha)$ and $b(\alpha)$ will respectively be the covariance and drift for the RBM associated with $\alpha$. 
For each $\alpha\in U$, we let $\sigma(\alpha)\in\M^{J\times J}$ denote the unique square root of $a(\alpha)$ so that $a(\alpha)=\sigma(\alpha)(\sigma(\alpha))^T$. 
Since $a$ is continuously differentiable,
	$$\sigma:U\mapsto\M^{J\times J}$$
is also continuously differentiable and we denote its Jacobian by $\sigma':U\mapsto\R^{J\times J}$.

We can now define an RBM in $G$ associated with $\alpha\in U$.

\begin{defn}\label{def:rbm}
Given $\{(d_i(\cdot),n_i),i\in\allN\}$, $b(\cdot)$, $\sigma(\cdot)$, $\alpha\in U$ and a $J$-dimensional $\{\F_t\}$-Brownian motion $\bm$ on $(\Omega,\F,\P)$, an RBM associated with parameter $\alpha$ and driving Brownian motion $\bm$ is a $J$-dimensional continuous $\{\F_t\}$-adapted process $\Z^\alpha=\{\Z^\alpha(t),t\geq0\}$ on $(\Omega,\F,\P)$ such that a.s.\ for all $t\geq0$, $\Z^\alpha(t)\in G$ and
	\be\label{eq:Z}\Z^\alpha(t)=\Z^\alpha(0)+b(\alpha)t+\sigma(\alpha)\bm(t)+\Y^\alpha(t),\ee
where $\Y^\alpha=\{\Y^\alpha(t),t\geq0\}$ is a $J$-dimensional continuous $\{\F_t\}$-adapted process that a.s.\ satisfies, for all $0\leq s<t<\infty$,
	$$\Y^\alpha(t)-\Y^\alpha(s)\in\conv\lsb\bigcup_{u\in(s,t]}d(\alpha,\Z^\alpha(u))\rsb.$$
\end{defn}

We refer to $\Y^\alpha$ as the \emph{constraining process} associated with the RBM $\Z^\alpha$. 
The following assumption imposes a linear independence condition on the directions of reflection, which will be used to decompose the constraining process $\Y^\alpha$ into its action along each face.

\begin{ass}\label{ass:independent}
For each $\alpha\in U$, $\{d_i(\alpha),i\in\allN\}$ is a set of linearly independent vectors.
\end{ass}

Define $R:U\mapsto\R^{J\times J}$ by
	\be\label{eq:R}R(\alpha):=\begin{pmatrix}d_1(\alpha)&\cdots&d_J(\alpha)\end{pmatrix},\ee
and let $R':U\mapsto\R^{J\times J}$ denote the Jacobian of $R$. Under Assumption \ref{ass:independent}, given a reflected Brownian motion $\Z^\alpha$ with associated constraining process $\Y^\alpha$, define the $J$-dimensional $\{\F_t\}$-adapted continuous process $\L^\alpha=\{\L^\alpha(t),t\geq0\}$ by 
\begin{equation}\label{eq:Lalpha}
	\L^\alpha(t):=(R(\alpha))^{-1}\Y^\alpha(t),\qquad t\geq0.
\end{equation}
We refer to $\L^\alpha$ as the \emph{local time process} corresponding to the RBM $\Z^\alpha$ because the $i$th component of $\L^\alpha(t)$ equals the local time that the RBM $\Z^\alpha$ spends in the $i$th face $F_i$ on the interval $[0,t]$.
We will make use of the following property of the local time process.

\begin{lem}
	[{\cite[Lemma 2.5]{Lipshutz2017}}]
	\label{lem:L}
	Suppose Assumption \ref{ass:independent} holds and $\Z^\alpha$ is an RBM with associated constraining process $\Y^\alpha$. Let $\L^\alpha$ be defined as in \eqref{eq:Lalpha}. 
	Then a.s.\ for each $i\in\allN$, the $i$th component of $\L^\alpha$ satisfies $\L^{\alpha,i}(0)=0$, $\L^{\alpha,i}$ is nondecreasing and $\L^{\alpha,i}$ can only increase when $\Z^\alpha$ lies in face $F_i$; that is,
		$$\int_0^\infty 1{\{\Z^\alpha(s)\in F_i\}}d\L^{\alpha,i}(s)=0.$$	
	Consequently, $\Y^\alpha$ is of bounded variation and so $\Z^\alpha$ is an $\{\F_t\}$-semimartingale.
\end{lem}

The next assumption states that for each $\alpha\in U$ the data $\{(d_i(\alpha),n_i),i\in\allN\}$ satisfies the geometric conditions introduced in \cite{Dupuis1991} to ensure Lipschitz continuity of the associated Skorokhod map. 
Given a convex set $B$, we let
	$$\nu_B(z):=\{\nu\in\sphere:\ip{\nu,y-z}\geq0\text{ for all }y\in B\}$$
denote the set of inward normal vectors to the set $B$ at $z\in\partial B$.

\begin{ass}\label{ass:setB}
For each $\alpha\in U$ there exists $\theta(\alpha)>0$ and a compact, convex, symmetric set $B^\alpha$ in $\R^J$ with $0\in(B^\alpha)^\circ$ such that for $i\in\allN$,
	\be\label{eq:setB}\lcb
	\begin{array}{c}
		z\in\partial B^\alpha\\
		|\ip{z,n_i}|<\theta(\alpha)
	\end{array}\rcb\qquad\Rightarrow\qquad\ip{\nu,d_i(\alpha)}=0\qquad\text{for all }\;\nu\in\nu_{B^\alpha}(z).\ee
Furthermore, $\alpha\mapsto B^\alpha$ is continuous in the Hausdorff metric defined in \eqref{eq:Hausdorff}.
\end{ass}

For existence, we will impose the following assumption, which requires that for each $\alpha\in U$ there exists a projection from $\R^J$ to $G$ that satisfies a certain condition with respect to the directions of reflection.

\begin{ass}\label{ass:projection}
For each $\alpha\in U$ there is a function $\pi^\alpha:\R^J\mapsto G$ satisfying $\pi^\alpha(x)=x$ for all $x\in G$ and $\pi^\alpha(x)-x\in d(\alpha,\pi^\alpha(x))$ for all $x\not\in G$.
\end{ass}

See the appendix of \cite{Lipshutz2017a} for an explicit expression for $\pi^\alpha$ in the case Assumptions \ref{ass:independent}, \ref{ass:setB} and \ref{ass:projection} hold.
 
We now state an easily verifiable algebraic sufficient condition on the reflection directions under which Assumptions \ref{ass:independent}, \ref{ass:setB} and \ref{ass:projection} hold.
The proof is deferred to the Appendix.
Let $N:=\begin{pmatrix}n_1&\cdots&n_J\end{pmatrix}$ denote the $J \times J$  matrix whose $i$th column is $n_i$, and recall that $E_J$ denotes the $J\times J$ identity matrix.

\begin{lem}
\label{lem-setB}
Suppose $N^TR(\alpha)$ is a (non-singular) $\mathcal{M}$-matrix for each $\alpha\in U$; that is, for each $\alpha\in U$, the matrix $Q(\alpha):=E_J-N^TR(\alpha)$ is non-negative whose spectral radius, denoted $\varrho(Q(\alpha))$, satisfies $\varrho(Q(\alpha))<1$.
Then Assumptions \ref{ass:independent}, \ref{ass:setB} and \ref{ass:projection} hold.
\end{lem}

\begin{remark}
RBMs in the non-negative orthant with reflection matrices that are $\mathcal{M}$-matrices arise as heavy traffic limits of single class open queueing networks \cite{Reiman1984}. 
Upon setting $N=E_J$ in Lemma \ref{lem-setB}, we see that if $G=\R_+^J$ and $R(\alpha)$ is an $\mathcal{M}$-matrix for each $\alpha\in U$, then Assumptions \ref{ass:independent}, \ref{ass:setB} and \ref{ass:projection} hold.
\end{remark}

We have the following theorem on the existence and pathwise uniqueness of an RBM. 

\begin{theorem}
	[{\cite[Theorem 4.3]{Ramanan2006}}]
	\label{thm:rbm}
	Suppose the data $\{(d_i(\cdot),n_i),i\in\allN\}$ satisfies Assumptions \ref{ass:setB} and \ref{ass:projection}. 
	Then given $\alpha\in U$, $x\in G$, and a $J$-dimensional $\{\F_t\}$-adapted Brownian motion $\bm$ on $(\Omega,\F,\P)$, there exists a pathwise unique RBM associated with $\alpha$ starting at $x$, which we denote by $\Z^{\alpha,x}$. 
	In other words, if $\Z^{\alpha,x}$ and $\wt\Z^{\alpha,x}$ are both RBMs associated with $\alpha$ and driving Brownian motion $\bm$ that almost surely satisfy $\Z^{\alpha,x}(0)=\wt\Z^{\alpha,x}(0)=x$, then almost surely $\Z^{\alpha,x}(t)=\wt\Z^{\alpha,x}(t)$ for all $t\ge0$. 
	Moreover, $\Z^{\alpha,x}$ is a strong Markov process.
\end{theorem}

\begin{remark}
	\label{rmk:pathwiseunique}
	The pathwise uniqueness of the RBM $\Z^{\alpha,x}$, along with \eqref{eq:Z} and \eqref{eq:Lalpha}, implies that the processes $\Y^{\alpha,x}$ and $\L^{\alpha,x}$ are also pathwise unique.
\end{remark}

Throughout this work, under Assumptions \ref{ass:setB} and \ref{ass:projection}, given $\alpha\in U$ and $x\in G$, we write $\Z^{\alpha,x}$ to denote the pathwise unique RBM associated with $\alpha$ starting at $x$; and we use $\Y^{\alpha,x}$ and $\L^{\alpha,x}$ (provided Assumption \ref{ass:independent} holds) to denote the corresponding constraining process and local time process, respectively. 

\subsection{The derivative process}\label{sec:derivativeprocess}

In this section we introduce the notion of a derivative process along an RBM which will be used in the next section to characterize pathwise derivatives of the RBM. 
The derivative process was introduced in \cite{Lipshutz2017}, where it was shown that the right continuous regularization of pathwise derivatives of reflected diffusions satisfy a linear constrained stochastic differential equation (with jumps) whose coefficients and directions of reflection depend on the state of the reflected diffusion. 
Solutions of the linear constrained stochastic differential equation are referred to as derivative processes.

In order to specify the domain of the derivative process, for each $x\in\partial G$, we define the intersection of hyperplanes
	\be\label{eq:Hx}\hyper_x:=\bigcap_{i\in\allN(x)}\lcb y\in\R^J:\ip{y,n_i}=0, \rcb\ee
where we recall the definition of $\allN(x)$ given in \eqref{eq:allNx}, and for each $x\in G^\circ$, we set $\hyper_x:=\R^J$.

\begin{defn}\label{def:de}
Suppose Assumption \ref{ass:independent} holds. 
Given $\alpha\in U$ and a $J$-dimensional $\{\F_t\}$-Brownian motion $\bm$ on $(\Omega,\F,\P)$, suppose $\Z^\alpha$ is an RBM associated with $\alpha$ and driving Brownian motion $\bm$. 
A derivative process along $\Z^\alpha$ is a $J$-dimensional RCLL $\{\F_t\}$-adapted process $\phib^\alpha=\{\phib^\alpha(t),t\geq0\}$ taking values in $\R^J$ such that a.s.\ for all $t\geq0$, $\phib^\alpha(t)\in \hyper_{\Z^\alpha(t)}$ and
	\be\label{eq:J}\phib^\alpha(t)=\phib^\alpha(0)+b'(\alpha) t+\sigma'(\alpha)\bm(t)+R'(\alpha)\L^\alpha(t)+\etab^\alpha(t),\ee
where $\etab^\alpha=\{\etab^\alpha(t),t\geq0\}$ is a $J$-dimensional RCLL $\{\F_t\}$-adapted process such that a.s.\ $\etab^\alpha(0)=0$ and for all $0\leq s<t<\infty$,
	$$\etab^\alpha(t)-\etab^\alpha(s)\in\spaan\lsb\bigcup_{u\in(s,t]}d(\alpha,\Z^\alpha(u))\rsb.$$
\end{defn}

\begin{remark}
In \cite{Lipshutz2017} (see also Theorem \ref{thm:pathwise} below) it was shown that a.s.\ the RBM $\Z^\alpha(\cdot)$ is differentiable with respect to $\alpha$ at each $t\ge0$, and the right-continuous regularization of its pathwise derivative is equal the derivative process $\phib^\alpha$ along $\Z^\alpha$.
From this perspective, it is natural to view  \eqref{eq:J} as a formal linearization of \eqref{eq:Z}, with $\phib^\alpha(t)$ and $R'(\alpha)\L^\alpha(t)+\etab^\alpha(t)$ serving as the appropriate linearizations of $\Z^\alpha(t)$ and $R(\alpha)\L^\alpha(t)$, respectively.
\end{remark}

Under the following assumption there exists a pathwise unique derivative process along an RBM.

\begin{ass}\label{ass:holder}
	There exists $\lip'<\infty$ and $\gamma\in(0,1]$ such that for all $\alpha,\wt\alpha\in U$,
	$$|b'(\alpha)-b'(\wt\alpha)|+\norm{\sigma'(\alpha)-\sigma'(\wt\alpha)}+\norm{R'(\alpha)-R'(\wt\alpha)}\le\lip'|\alpha-\wt\alpha|^\gamma.$$
\end{ass}

\begin{theorem}
[{\cite[Corollary 3.15]{Lipshutz2017}}]
\label{thm:derivativeprocess}
Suppose the data $\{(d_i(\cdot),n_i),i\in\allN\}$ satisfies Assumptions \ref{ass:independent}, \ref{ass:setB} and \ref{ass:projection} and $b(\cdot)$, $\sigma(\cdot)$ and $R(\cdot)$ satisfy Assumption \ref{ass:holder}. Then given $\alpha\in U$, $x\in G$, $y\in \hyper_x$ and a $J$-dimensional $\{\F_t\}$-adapted Brownian motion $\bm$ on $(\Omega,\F,\P)$, there exists a pathwise unique derivative process along {the RBM} $\Z^{\alpha,x}$ starting at $y$, which we denote by $\phib^{\alpha,\xi}$, where $\xi:=(x,y)$. 
In other words, if $\phib^{\alpha,\xi}$ and $\wt\phib^{\alpha,\xi}$ are both derivative processes along the RBM $\Z^{\alpha,x}$ that satisfy $\phib^{\alpha,\xi}(0)=\wt\phib^{\alpha,\xi}(0)=y$, then almost surely $\phib^{\alpha,\xi}(t)=\wt\phib^{\alpha,\xi}(t)$ for all $t\ge0$. 
\end{theorem}

\subsection{Pathwise derivatives of RBMs}\label{sec:pathwise}

{We now introduce} the notion of a pathwise (directional) derivative of an RBM and present results on their existence and characterization. 
Roughly speaking, a pathwise derivative of an RBM is a process that characterizes sensitivities of sample paths of an RBM on finite time intervals. 
They were shown to exist and were characterized in terms of the derivative process in \cite{Lipshutz2017}. 
To be precise, given $\alpha\in U$, $x\in G$, $y\in \hyper_x$ and $\ve>0$ such that $\alpha+\ve\in U$ and $x+\ve y\in G$, define the $J$-dimensional continuous process $\nabla_y^\ve\Z^{\alpha,x}=\{\nabla_y^\ve\Z^{\alpha,x}(t),t\geq0\}$ on $(\Omega,\F,\P)$ by
	$$\nabla_y^\ve\Z^{\alpha,x}(t):=\frac{\Z^{\alpha+\ve,x+\ve y}(t)-\Z^{\alpha,x}(t)}{\ve},\qquad t\geq0.$$
The pathwise derivative of $\Z^{\alpha,x}$ in the direction $y$ evaluated at time $t\geq0$ is defined to be the pointwise limit of $\nabla_y^\ve\Z^{\alpha,x}(t)$ as $\ve\downarrow0$.

For the following theorem, let $\dlr(\R^J)$ denote the set of functions $f:[0,\infty)\mapsto\R^J$ that have finite left limits at each $t>0$ and finite right limits at each $t\geq0$. 
Let
	$$\U:=\{x\in\partial G:|\allN(x)|\geq2\}$$
denote the nonsmooth part of the boundary of $G$.

\begin{theorem}[{\cite[Theorem 3.18]{Lipshutz2017}}]
\label{thm:pathwise}
Suppose the data $\{(d_i(\cdot),n_i),i\in\allN\}$ satisfies Assumptions \ref{ass:independent}, \ref{ass:setB} and \ref{ass:projection} and $b(\cdot)$, $\sigma(\cdot)$ and $R(\cdot)$ satisfy Assumption \ref{ass:holder}. Let $\alpha\in U$ and $x\in G$. Then a.s.\ the following hold for all $y\in \hyper_x$:
\begin{itemize}
	\item[(i)] For each $t\geq0$, the following limit exists: $\nabla_{y}\Z^{\alpha,x}(t):=\lim_{\ve\downarrow0}\nabla_{y}^\ve\Z^{\alpha,x}(t)$.
	\item[(ii)] The process $\nabla_{y}\Z^{\alpha,x}=\{\nabla_{y}\Z^{\alpha,x}(t),t\geq0\}$ takes values in $\dlr(\R^J)$ and is continuous at all $t>0$ such that $\Z^{\alpha,x}(t)\in G^\circ\cup\U$.
	\item[(iii)] The right continuous regularization of $\nabla_y\Z^{\alpha,x}$ is equal to the derivative process along $\Z^{\alpha,x}$ starting at $y$; that is, $\nabla_{y}\Z^{\alpha,x}(t+)=\phib^{\alpha,\xi}(t)$ for all $t\geq0$, where $\xi:=(x,y)$.
\end{itemize}
\end{theorem}

\begin{remark}
\label{rem-dirder}
It should be noted that the existence of pathwise directional derivatives of RBMs for a sub-class of reflection matrices  $R(\alpha)$ that are  so-called  ${\mathcal M}$-matrices (or of Harrison-Reiman type) was also established in \cite{ManRam00}.  However, the characterization given above of the right-continuous regularizations of these pathwise derivatives, besides being applicable to a more general class of reflection matrices, also  has several useful linearity and continuity properties that will be particularly useful in the ergodicity analysis.
\end{remark}

	As a corollary of Theorem \ref{thm:pathwise}, we have the following result.
	Let $\zeta_1:G\mapsto\R$ and $\zeta_2:G\mapsto\R$ be continuously differentiable functions with bounded first partial derivatives.
	For $t>0$ define $\Theta_t:U\mapsto\R$ by
		$$\Theta_t(\alpha):=\E\lsb\int_0^t\zeta_1(\Z^{\alpha,x}(s))ds+\zeta_2(\Z^{\alpha,x}(t))\rsb,\qquad \alpha\in U.$$

\begin{cor}[{\cite[Corollary 3.16]{Lipshutz2017}}]
\label{cor:pathwise}
Suppose the assumptions stated in Theorem \ref{thm:pathwise} hold. 
Then for each $t>0$ and $\alpha\in U$, $\Theta_t$ is differentiable at $\alpha$ and its derivative at $\alpha$ satisfies
	$$\Theta_t'(\alpha)=\E\lsb\int_0^t\zeta_1'(\Z^{\alpha,x}(s))\phib^{\alpha,\xi}(s)ds+\zeta_2'(\Z^{\alpha,x}(t))\phib^{\alpha,\xi}(t)\rsb.$$
\end{cor} 

\section{Main results}
\label{sec:main}

In this section we present our main results. 
In Section \ref{sec:joint} we {establish} the Feller Markov property for the joint process consisting of an RBM and its derivative process. 
In Section \ref{sec:jointstable} we {identify a sufficient condition under which} the joint process is stable and has a unique stationary distribution. 
In Section \ref{sec:sensitivityinvariant} we show that sensitivities of the stationary distribution of the RBM can be expressed in terms of the stationary distribution of the joint process.

\subsection{Feller Markov property of the joint process}
\label{sec:joint}

Given $\alpha\in U$ an RBM $\Z^\alpha$ associated with $\alpha$ and a derivative process $\phib^\alpha$ along $\Z^\alpha$, we refer to the $\{\F_t\}$-adapted process $\Xib^\alpha=\{\Xib^\alpha(t),t\geq0\}$ on $(\Omega,\F,\P)$, defined by
	$$\Xib^\alpha(t):=(\Z^\alpha(t),\phib^\alpha(t)),\qquad t\geq0,$$
as the joint RBM-derivative process, or joint process for short, associated with $\alpha$. Define
	\be\label{eq:state}\state:=\bigcup_{x\in G}\lb\{x\}\times \hyper_x\rb\subseteq G\times\R^J.\ee
It follows from Definitions \ref{def:rbm} and \ref{def:de} that a.s.\ the joint process $\Xib^\alpha$ takes values in $\state$. 
By Theorems \ref{thm:rbm} and \ref{thm:derivativeprocess}, given $\xi=(x,y)\in\state$, there exists a pathwise unique joint process associated with $\alpha$ {that has initial condition $\xi$}, which we denote by $\Xib^{\alpha,\xi}=(\Z^{\alpha,x},\phib^{\alpha,\xi})$.

\begin{remark}
Since $\state$ is a subset of the Polish space $G\times\R^{J}$, $\state$ is a separable metric space; however, $\state$ is not closed and hence is not a Polish space. In particular, the closure of $\state$ is equal to $G\times\R^{J}$. 
This follows because $G$ is closed and convex with nonempty interior, and $\hyper_x=\R^{J}$ for all $x\in G^\circ$.
\end{remark}

For $\alpha\in U$ we can define a family 
	$$\lcb P_t^\alpha\rcb=\lcb P_t^\alpha(\xi,A),t\geq0,\xi\in\state,A\in\B(\state)\rcb$$ 
of transition functions on $\state$ associated with $\{\Xib^{\alpha,\xi}(t),t\geq0,\xi\in\state\}$ by
\be\label{eq:Pt}P_t^\alpha(\xi,A):=\P(\Xib^{\alpha,\xi}(t)\in A),\qquad \xi\in\state,\; A\in\B(\state).\ee

Recall that the family $\{P_t^\alpha\}$ is Markovian if for every bounded measurable function $\zeta:\state\mapsto\R$ and $s,t>0$, $\E[\zeta(\Xi(s+t))|\F_s]=\E[\zeta(\Xi(s+t))|\Xi(s)]$ holds, and the family $\{P_t^\alpha\}$ is Feller continuous if for every bounded continuous function $\zeta:\state\mapsto\R$ and $t>0$, the mapping $\xi\mapsto(P_t^\alpha\zeta)(\xi):=\int_\state\zeta(\tilde\xi)P_t^\alpha(\xi,d\tilde\xi)$ is continuous. 

\begin{theorem}\label{thm:jointmarkov}
Suppose the data $\{(d_i(\cdot),n_i),i\in\allN\}$ satisfies Assumptions \ref{ass:independent}, \ref{ass:setB} and \ref{ass:projection}, and the coefficients $b(\cdot)$, $\sigma(\cdot)$ and $R(\cdot)$ satisfy Assumption \ref{ass:holder}. 
For each $\alpha\in U$, the family of transition functions $\{P_t^\alpha\}$ is Markovian and Feller continuous.
\end{theorem}

The Feller continuity of the joint process is used in the proof of existence of a stationary distribution for the joint process (see Theorem \ref{thm:stable} below).
We establish Theorem \ref{thm:jointmarkov} in Section \ref{sec:Feller}.
Its proof relies on properties of solutions to the so-called SP and derivative problem introduced in Sections \ref{sec:SP} and \ref{sec:dp}, respectively, which are useful for the construction and analysis of the RBM and derivative process.  Key ingredients of the proof include continuity results of a certain map called the derivative map (see Propositions \ref{prop:dmlip} and \ref{prop:dpexistence}).

\subsection{Ergodicity of the joint process}
\label{sec:jointstable}

Given a probability measure $\mu$ on $(\state,\B(\state))$, define, for all $t\geq0$ and $A\in\B(\state)$,
	$$(\mu P_t^\alpha)(A):=\int_{\state} P_t^\alpha(\xi,A)\mu(d\xi).$$
For a measurable and integrable (with respect to $\mu$) function $g:\state\mapsto\R$, we write
	$$\mu(g):=\int_\state g(\xi)\mu(d\xi).$$

\begin{defn}\label{def:invariant}
	A stationary distribution of the joint process associated with $\alpha$ is a probability measure $\mu^\alpha$ on $(\state,\B(\state))$ such that $\mu^\alpha(A)=(\mu^\alpha P_t^\alpha)(A)$ for all $A\in\B(\state)$ and $t\geq0$.
\end{defn}

The following assumption, introduced by Budhiraja and Dupuis in \cite{Budhiraja1999}, states that for each $\alpha\in U$, the drift vector $b(\alpha)$ lies in the interior of a certain cone determined by the directions of reflection. 
The condition implies that all trajectories of a related deterministic model are attracted to the origin, which was shown by Dupuis and Williams \cite{Dupuis1994} to be a sufficient condition for the RBM to be positive recurrent and have a unique stationary distribution.
We also mention the work of Atar, Budhiraja and Dupuis \cite{Atar2001} who show that a related condition is sufficient for positive recurrence and existence of a unique stationary distribution for reflected diffusions; and the work of Budhiraja and Lee \cite{Budhiraja2007} who prove geometric ergodicity of the RBM under the following assumption (and of a reflected diffusion under a related condition).
We now state the assumption and the ergodicity result. 
Let $\C^\alpha$ denote the cone in $\R^J$ with vertex at the origin defined by
	\be\label{eq:Calpha}\C^\alpha:=\conv\lb\{-d_i(\alpha),i\in\allN\}\rb.\ee

\begin{ass}\label{ass:stable}
	For each $\alpha\in U$, $b(\alpha)\in(\C^\alpha)^\circ$.
\end{ass}

\begin{theorem}[{\cite[Theorem 3.8]{Budhiraja1999} \& \cite[Theorem 2.6]{Dupuis1994}}]
	\label{thm:RBMstable}
 Suppose the data $\{(d_i(\cdot),n_i),i\in\allN\}$ satisfies Assumptions \ref{ass:independent}, \ref{ass:setB} and \ref{ass:projection}, and the coefficients $b(\cdot)$, $\sigma(\cdot)$ and $R(\cdot)$ satisfy Assumptions \ref{ass:holder} and \ref{ass:stable}. 
	Then for each $\alpha\in U$, $\Z^\alpha$ is positive recurrent and has a unique stationary distribution.
\end{theorem}

The following is the main result of this section.

\begin{theorem}\label{thm:stable}
Suppose the data $\{(d_i(\cdot),n_i),i\in\allN\}$ satisfies Assumptions \ref{ass:independent}, \ref{ass:setB} and \ref{ass:projection}, and the coefficients $b(\cdot)$, $\sigma(\cdot)$ and $R(\cdot)$ satisfy Assumptions \ref{ass:holder} and \ref{ass:stable}. 
	For each $\alpha\in U$ there exists a unique stationary distribution of the joint process $\Xib^\alpha=(\Z^\alpha,\phib^\alpha)$ associated with $\alpha$.
\end{theorem}

\begin{remark}
	Throughout the remainder of this work we let $\Xib^\alpha(\infty)=(\Z^\alpha(\infty),\phib^\alpha(\infty))$ denote a random variable on $(\Omega,\F,\P)$ taking values in $\state$ that is independent of the Brownian motion $\bm$ and equal in distribution to the unique stationary distribution of the joint process.
\end{remark}

The proof of Theorem \ref{thm:stable}, which is presented  in Section \ref{sec:proofstable},  
incorporates several novel features.   It relies on certain contraction properties for
solutions of the derivative problem that are established in Section \ref{sec:projection}.
In turn, these properties are showed in Section \ref{sec:dpstable} to imply  corresponding contraction properties for the derivative process that hold after the related RBM visits every face of the cone, and the
latter event is shown to happen infinitely often in Section \ref{sec:rbmestimates}.
Existence of the stationary distribution is then proved in Section \ref{sec:stable} by constructing
a Lyapunov function that  involves the norm (associated with the set $B$ of Assumption \ref{ass:setB}) of the derivative process, and is used to show stability of the joint process.
Furthermore, the stability is shown to be uniform, in a sense, over $\alpha$ in compact subsets of $U$, which will be useful for proving our next main result.
The proof of uniqueness of the stationary distribution (Theorem \ref{thm:unique}) is somewhat tricky due to the
degeneracy of the $2J$-dimensional joint process (driven by a $J$-dimensional Brownian motion).
The proof uses an asymptotic coupling argument and relies on the linearity of the derivative process as well as the above contraction properties of the derivative process.

\subsection{Sensitivity analysis for the stationary distribution of an RBM}\label{sec:sensitivityinvariant}

In this section we present our main result on sensitivities of the stationary distribution of an RBM. 
Let $f:G\mapsto\R$ be continuous differentiable with bounded and continuous Jacobian $f':G\mapsto\R^{1\times J}$, and define $F:U\mapsto\R$ by
	$$F(\alpha):=\E\lsb f(\Z^\alpha(\infty))\rsb,\qquad\alpha\in U.$$

\begin{theorem}\label{thm:sensitivity}
 Suppose the data $\{(d_i(\cdot),n_i),i\in\allN\}$ satisfies Assumptions \ref{ass:independent}, \ref{ass:setB} and \ref{ass:projection}, and the coefficients $b(\cdot)$, $\sigma(\cdot)$ and $R(\cdot)$ satisfy Assumptions \ref{ass:holder} and \ref{ass:stable}. 
	Then for almost every $\alpha\in U$ the function $F(\cdot)$ is differentiable and its Jacobian satisfies
	\be\label{eq:sensitivity}F'(\alpha)=\E\lsb f'(\Z^\alpha(\infty))\phib^\alpha(\infty)\rsb.\ee
\end{theorem}

The proof of Theorem \ref{thm:sensitivity}, which uses uniform stability properties for the joint process established in Section \ref{sec:stable} along with standard real analysis arguments, is given in Section \ref{sec:interchange}. 

\section{The Skorokhod problem, derivative problem and contraction properties}\label{sec:sdm}

In this section we carry out a deterministic analysis of solutions to the SP and of solutions to the derivative problem. 
The results in this section are used in Sections \ref{sec:Feller}--\ref{sec:stable} to prove our main results.
In Sections \ref{sec:SP} and \ref{sec:dp} we state the SP, the derivative problem and review some relevant properties. 
In the remaining sections we prove new stability properties related to solutions of the SP and derivative problem.

Throughout this section we assume, without restatement, that the data $\{(d_i(\cdot),n_i),i\in\allN\}$ satisfies Assumptions \ref{ass:independent}, \ref{ass:setB} and \ref{ass:projection}. 

\subsection{The Skorokhod reflection problem}\label{sec:SP}

The SP (in a polyhedral cone) provides an axiomatic framework to constrain a path taking values in Euclidean space to a polyhedral cone.
Throughout this section we fix $\alpha\in U$.

\begin{defn}\label{def:SP}
Given $\x\in\cts_G(\R^J)$, a pair $(\z,\y)\in\cts(\R^J)\times\cts(\R^J)$ is a solution to the SP $\{(d_i(\alpha),n_i),i\in\allN\}$ for $\x$ if $\z(0)=\x(0)$, and if for all $t\geq0$, the following properties hold:
\begin{itemize}
	\item[1.] $\z(t)=\x(t)+\y(t)$;
	\item[2.] $\z(t)\in G$;
	\item[3.] for every $s\in[0,t)$,
		$$\y(t)-\y(s)\in\conv\lsb\bigcup_{u\in(s,t]}d(\alpha,\z(u))\rsb.$$
\end{itemize}
If there exists a unique solution $(\z,\y)$ to the SP $\{(d_i(\alpha),n_i),i\in\allN\}$ for $\x$, we write $\z=\sm^\alpha(\x)$ and refer to $\sm^\alpha$ as the SM associated with the SP $\{(d_i(\alpha),n_i),i\in\allN\}$.
\end{defn}

\begin{remark}
In the standard formulation of the SP (see, e.g., \cite[Definition 1.1]{Ramanan2006}), instead of condition 3 of Definition \ref{def:SP}, the constraining processes $\y$ is assumed to have bounded variation (i.e., $|\y|(t)<\infty$ for all $t\ge0$) and satisfy the following conditions for all $t\ge0$:
	$$|\y|(t)=\int_{[0,t]}1\{\z(s)\in\partial G\}d|\y|(s),$$
and there exists a measurable function $\gamma:[0,\infty)\mapsto\sphere$ such that $\gamma(s)\in d(\alpha,\z(s))$ ($d\abs{\y}$-almost everywhere) and
	$$\y(t)=\int_{[0,t]}\gamma(s)d|\y|(s).$$
Condition 3 was introduced in \cite{Ramanan2006} to allow for constraining processes with unbounded variation, and this generalization is referred to as the \textit{extended Skorokhod problem}.
Under the linear independence condition on the directions of {reflection} stated in Assumption \ref{ass:independent}, the constraining term $\y$ in Definition \ref{def:SP} must be of bounded variation, and it follows from \cite[Theorem 1.3]{Ramanan2006} that a pair $(\z,\y)$ satisfying Definition \ref{def:SP} for $\x\in\cts_G(\R^J)$ also satisfies the standard formulation of the SP stated in \cite[Definition 1.1]{Ramanan2006}.
\end{remark}

\begin{remark}\label{rmk:ZsmX}
Given a $J$-dimensional $\{\F_t\}$-Brownian motion $\bm$ on $(\Omega,\F,\P)$ and an RBM $\Z^\alpha$ with driving Brownian motion $\bm$, define the $J$-dimensional continuous $\{\F_t\}$-adapted process $\X^\alpha=\{\X^\alpha(t),t\geq0\}$ on $(\Omega,\F,\P)$ by
	$$\X^\alpha(t):= \Z^\alpha(0)+b(\alpha)t+\sigma(\alpha)\bm(t),\qquad t\geq0.$$
Then it follows from the conditions in Definition \ref{def:rbm} and the definition to the SP that a.s.\ $(\Z^\alpha,\Y^\alpha)$ is a solution to the SP $\{(d_i(\alpha),n_i),i\in\allN\}$ for $\X^\alpha$.
\end{remark}

{ The first result we state is a useful}
 time-shift property of the SP. 
Given $\x\in\cts_G(\R^J)$ suppose $(\z,\y)$ is a solution to the SP $\{(d_i(\alpha),n_i),i\in\allN\}$ for $\x$. 
For $s\geq0$ recall the shift operator defined in \eqref{eq:shift} as $(\Theta_s f)(t):=f(s+t)-f(s)$ for $t\ge0$. 
Define $\z^s\in\cts(G)$, $\y^s\in\cts(\R^J)$ and $\x^s\in\cts_G(\R^J)$ by
\begin{align}\label{eq:zS}
	\z^s(\cdot)&:=\z(s+\cdot),\\ \label{eq:yS}
	\y^s(\cdot)&:=(\Theta_s\y)(\cdot),\\ \label{eq:xS}
	\x^s(\cdot)&:=\z(s)+(\Theta_s\x)(\cdot).
\end{align}

\begin{lem}
[{\cite[Lemma 2.3]{Ramanan2006}}]
\label{lem:SPtimeshift}
Let $(\z,\y)$ be a solution to the SP $\{(d_i(\alpha),n_i),i\in\allN\}$ for $\x\in\cts_G(\R^J)$. 
Then for $s\geq0$, $(\z^s,\y^s)$ is a solution to the SP for $\x^s$.
\end{lem}

The next result { we state concerns} existence and uniqueness of solutions to the SP as well as a Lipschitz continuity property for solutions of the SP.

\begin{prop}
[{\cite[Theorem 2.12]{Lipshutz2018} \& \cite[Theorem 3.3]{Ramanan2006}}]
\label{prop:SP}
Given $\x\in\cts_G(\R^J)$ there exists a unique solution $(\z,\y)$ to the SP $\{(d_i(\alpha),n_i),i\in\allN\}$ for $\x$. 
Furthermore, there exists $\lip_\sm(\alpha)<\infty$ such that if $(\z_k,\y_k)$ is a solution to the SP $\{(d_i(\alpha),n_i),i\in\allN\}$ for $\x_k\in\cts_G(\R^J)$, for $k=1,2$, then for all $t\geq0$,
\begin{align*}
	\sup_{s\in[0,t]}|\z_1(s)-\z_2(s)|+\sup_{s\in[0,t]}|\y_1(s)-\y_2(s)|&\leq\lip_\sm(\alpha)\sup_{s\in[0,t]}|\x_1(s)-\x_2(s)|.
\end{align*}
\end{prop}

The following lemma will be useful for proving bounds that are uniform over $\alpha$ in compact subsets of $U$. 

\begin{lem}
  \label{lem:lipcont}
	The constant $\lip_\sm(\alpha)$ in Proposition \ref{prop:SP} can be chosen to be continuous in $\alpha\in U$.
\end{lem}
\begin{proof}
	Let $\x_1,\x_2\in\cts_G(\R^J)$.
	For $k=1,2$ and $\alpha\in U$ let $(\z_k^\alpha,\y_k^\alpha)$ denote the unique solution to the SP $\{(d_i(\alpha),n_i),i\in\allN\}$ for $\x_k$, and set $\ell_k^\alpha(\cdot):=(R(\alpha))^{-1}\y_k^\alpha(\cdot)$.
	Fix $\alpha_0\in U$.
	For $\alpha\in U$, define
		$$\x_k^\alpha(t):=\x_k(t)+(R(\alpha)-R(\alpha_0))\ell_k^\alpha(t)$$
	and
		$$\tilde\y_k^\alpha(t):=R(\alpha_0)\ell_k^\alpha(t)=R(\alpha_0)(R(\alpha))^{-1}\y_k^\alpha(t).$$
	Then $(\z_k^\alpha,\tilde\y_k^\alpha)$ solves the SP $\{(d_i(\alpha_0),n_i),i\in\allN\}$ for $\x_k^\alpha$.
	It follows from the Lipschitz continuity of the SM (Proposition \ref{prop:SP}) that
	\begin{align*}
		\sup_{s\in[0,t]}|\x_1^\alpha(s)-\x_2^\alpha(s)|&\le\sup_{s\in[0,t]}|\x_1(s)-\x_2(s)|\\
		&\qquad+\norm{R(\alpha)-R(\alpha_0)}\norm{(R(\alpha_0))^{-1}}\sup_{s\in[0,t]}\abs{\tilde\y_1^\alpha(s)-\tilde\y_2^\alpha(s)}\\
		&\le\sup_{s\in[0,t]}|\x_1(s)-\x_2(s)|\\
		&\qquad+\lip_\sm(\alpha_0)\norm{R(\alpha)-R(\alpha_0)}\norm{(R(\alpha_0))^{-1}}\sup_{s\in[0,t]}\abs{\x_1^\alpha(s)-\x_2^\alpha(s)}, 
	\end{align*}
        {which, on rearranging, yields } 
	\begin{align*}
		\sup_{s\in[0,t]}|\x_1^\alpha(s)-\x_2^\alpha(s)|&\le\frac{1}{1-\lip_\sm(\alpha_0)\norm{R(\alpha)-R(\alpha_0)}\norm{(R(\alpha_0))^{-1}}}\sup_{s\in[0,t]}|\x_1(s)-\x_2(s)|.
	\end{align*}
       {By the continuity of $R(\cdot)$, given $\ve  > 0$, there exists
        $\delta > 0$, such that if $|\alpha - \alpha_0| < \delta$, then
        the constant on the right-hand side of the last inequality can be made
        less than $1 + \ve/\kappa_{\Gamma}(\alpha_0)$. 
	Together with  the Lipschitz continuity of the SM,
        this implies for all such $\alpha$, }
	\begin{align*}
	  \sup_{s\in[0,t]}|\z_1^\alpha(s)-\z_2^\alpha(s)|&\le
          (\kappa_{\Gamma}(\alpha_0) + \ve) \sup_{s\in[0,t]}|\x_1(s)-\x_2(s)|.
       \end{align*}   
	Since $\y_k^\alpha=\x_k-\z_k^\alpha$ and $R(\cdot)$ is continuous, we see that $\lip_\sm(\cdot)$ can be chosen to be continuous at $\alpha_0$, thus completing the proof.
\end{proof}

\begin{remark}
\label{rmk:lipcontinuous}
Throughout the remainder of this work we assume $\lip_\sm(\cdot)$ is continuous.
\end{remark}

We close this section by stating a slightly stronger version of the so-called \emph{boundary jitter property} that was introduced in \cite[Definition 3.1]{Lipshutz2018}. 
The boundary jitter property plays a crucial role in characterizing directional derivatives of the SP (see \cite[Theorem 3.11]{Lipshutz2018}). 
The stronger version 2' of condition 2 is used in \cite{Lipshutz2017a} to prove a continuity property of the derivative map stated in Proposition \ref{prop:dmcontinuous} below. { The latter}  is used in the next section to prove the joint process is Feller continuous.  { Recall that ${\mathcal N}$ is the set of non-smooth points in the boundary $\partial G$.} 

\begin{defn}\label{def:jitter}
A pair $(\z,\y)\in\cts(G)\times\cts(\R^J)$ is said to satisfy the boundary jitter property if the following conditions hold:
\begin{itemize}
	\item[1.] If $\z(t)\in\partial G$ for some $t\ge0$, then $\y$ is nonconstant on $(t_1^+,t_2)$ for all $t_1<t<t_2$.
	\item[2.'] $\z$ does not spend positive Lebesgue time in the boundary $\partial G$; that is,
		$$\int_0^\infty1\{\z(s)\in\partial G\}ds=0.$$
	\item[3.] If $\z(t)\in\U$ for some $t>0$, then for each $i\in\allN(\z(t))$ and all $\delta\in(0,t)$, there exists $s\in(t-\delta,t)$ such that $\allN(\z(s))=\{i\}$.
	\item[4.] If $\z(0)\in\U$, then for each $i\in\allN(\z(0))$ and all $\delta>0$, there exists $s\in(0,\delta)$ such that $\allN(\z(s))=\{i\}$.
\end{itemize}
\end{defn}

\begin{remark}
Condition 2 of \cite[Definition 3.1]{Lipshutz2018} only requires that $\z$ does not spend positive Lebesgue time in the nonsmooth part of the boundary $\U$.
\end{remark}

\begin{prop}[{\cite[Proposition 6.8]{Lipshutz2017a}}]
\label{prop:jitter}
A.s.\ $(\Z^\alpha,\Y^\alpha)$ satisfies the boundary jitter property.
\end{prop}

\subsection{The derivative problem}\label{sec:dp}

The derivative problem was first introduced in \cite[Definition 3.4]{Lipshutz2018} as an axiomatic framework for studying directional derivatives of the SM.
Throughout this section we fix $\alpha\in U$.

\begin{defn}
\label{def:dp}
Given $\x\in\cts_G(\R^J)$, suppose $(\z,\y)$ is a solution to the SP $\{(d_i(\alpha),n_i),i\in\allN\}$ for $\x$. 
Let $\psi\in\dr(\R^J)$. 
Then $(\phi,\eta)\in\dr(\R^J)\times\dr(\R^J)$ is a solution to the derivative problem along $\z$ for $\psi$ if $\eta(0)\in\spaan[d(\alpha,\z(0))]$ and if for all $t\geq0$, the following conditions hold:
\begin{itemize}
	\item[1.] $\phi(t)=\psi(t)+\eta(t)$;
	\item[2.] $\phi(t)\in \hyper_{\z(t)}$;
	\item[3.] for all $s\in[0,t)$,
		$$\eta(t)-\eta(s)\in\spaan\lsb\bigcup_{u\in(s,t]}d(\alpha,\z(u))\rsb.$$
\end{itemize}
If there exists a unique solution $(\phi,\eta)$ to the derivative problem along $\z$ for $\psi$, we write $\phi=\dm_\z^\alpha(\psi)$ and refer to $\dm_\z^\alpha$ as the derivative map associated with $\z$.
\end{defn}

\begin{remark}
\label{rmk:derivativeprocessbeta}
Given a derivative process $\phib^\alpha$ along an RBM $\Z^\alpha$, let $\etab^\alpha$ be as in Definition \ref{def:de} and define the $\{\F_t\}$-adapted continuous process $\psib^\alpha=\{\psib^\alpha(t),t\geq0\}$ taking values in $\R^J$ by
	\be\label{eq:psib}\psib^\alpha(t):=\phib^\alpha(0)+b'(\alpha)t+\sigma'(\alpha)\bm(t)+R'(\alpha)\L^\alpha(t),\qquad t\geq0,\ee
where we recall that $\Y^\alpha$ is the constraining process introduced in Definition \ref{def:rbm} and $\L^\alpha(\cdot)=R^{-1}(\alpha)\Y^\alpha(\cdot)$.
It follows from Definition \ref{def:de} and the definition of the derivative problem that a.s.\ the pair $(\phib^\alpha,\etab^\alpha)$ is a solution to the derivative problem along $\Z^\alpha$ for $\psib^\alpha$.
\end{remark}

We now state some useful properties of the derivative problem that were established in \cite{Lipshutz2018}. 
The first result states that the derivative map is linear.

\begin{lem}[{\cite[Lemma 5.1]{Lipshutz2018}}]
\label{lem:dmlinear}
Let $(\z,\y)$ be the solution to the SP $\{(d_i(\alpha),n_i),i\in\allN\}$ for $\x\in\cts_G(\R^J)$. 
Let $\psi_k\in\cts(\R^J)$ and suppose $\dm_\z^\alpha(\psi_k)$ is well defined, for $k=1,2$. 
Then $\dm_\z^\alpha(r_1\psi_1+r_2\psi_2)$ is well defined and equal to $r_1\dm_\z^\alpha(\psi_1)+r_2\dm_\z^\alpha(\psi_2)$ for all $r_1,r_2\in\R$.
\end{lem}

Our next result is a time-shift property of the derivative problem. 
Given $\x\in\cts_G$ and $\psi\in\cts$, suppose $(\z,\y)$ is a solution to the SP $\{(d_i(\alpha),n_i),i\in\allN\}$ for $\x$ and $(\phi,\eta)$ is a solution to the derivative problem along $\z$ for $\psi$. 
Let $s\geq0$ and recall the shift operator defined in \eqref{eq:shift}. 
Define $\z^s\in\cts(G)$ as in \eqref{eq:zS} and define $\phi^s,\eta^s,\psi^s\in\cts(\R^J)$ by
\begin{align}\label{eq:phis}
	\phi^s(\cdot)&:=\phi(s+\cdot),\\ \label{eq:etas}
	\eta^s(\cdot)&:=(\Theta_s\eta)(\cdot),\\ \label{eq:psis}
	\psi^s(\cdot)&:=\phi(s)+(\Theta_s\psi)(\cdot).
\end{align}

\begin{lem}[{\cite[Lemma 5.2]{Lipshutz2018}}]
\label{lem:dptimeshift}
Let $(\z,\y)$ be a solution to the SP $\{(d_i(\alpha),n_i),i\in\allN\}$ for $\x\in\cts_G$. 
Suppose $(\phi,\eta)$ is a solution to the derivative problem along $\z$ for $\psi$. 
Then for each $s\geq0$, $(\phi^s,\eta^s)$ is a solution to the derivative problem along $\z^s$ for $\psi^s$.
\end{lem}

The following proposition states a Lipschitz continuity property of the derivative map.

\begin{prop}[{\cite[Theorem 5.4]{Lipshutz2018}}]
\label{prop:dmlip}
There exists $\lip_\dm(\alpha)<\infty$ such that if $(\z,\y)$ is a solution to the SP $\{(d_i(\alpha),n_i),i\in\allN\}$ for $\x\in\cts_G(\R^J)$, $(\phi_1,\eta_1)$ is a solution to the derivative problem along $\z$ for $\psi_1\in\cts(\R^J)$, and $(\phi_2,\eta_2)$ is a solution to the derivative problem along $\z$ for $\psi_2\in\cts(\R^J)$, then for all $t\geq0$,
	\be\sup_{s\in[0,t]}|\phi_1(s)-\phi_2(s)|\leq\lip_\dm(\alpha)\sup_{s\in[0,t]}|\psi_1(s)-\psi_2(s)|.\ee
\end{prop}

The next result states that the derivative problem is well defined along $\z$ provided $(\z,\y)$ satisfies the boundary jitter property. 
The proposition follows from \cite[Theorem 3.11]{Lipshutz2018} and the fact that, by Assumption \ref{ass:independent} and \cite[Lemma 8.2]{Lipshutz2018},
	\be\label{eq:W}\mathcal{W}^\alpha:=\{x\in\U:\spaan(d(\alpha,x)\cup \hyper_x)\neq\R^J\}=\emptyset.\ee

\begin{prop}
[{\cite[Theorem 3.11]{Lipshutz2018}}]
\label{prop:dpexistence}
Let $(\z,\y)$ be a solution to the SP $\{(d_i(\alpha),n_i),i\in\allN\}$ for $\x\in\cts_G$. 
Suppose $(\z,\y)$ satisfies the boundary jitter property (Definition \ref{def:jitter}). 
Then for all $\psi\in\cts(\R^J)$ there exists a unique solution $(\phi,\eta)$ to the derivative problem along $\z$ for $\psi$.
\end{prop}

The next result states a continuity result for the derivative map that is used in the proof that the joint process is a Feller continuous Markov process (see Section \ref{sec:Feller}). 
The proposition is a version of \cite[Theorem 6.15]{Lipshutz2017a} written for the case that $\x_k$ and $\psi_k$ are continuous for each $k\in\N$.

\begin{prop}
[{\cite[Theorem 6.15]{Lipshutz2017a}}]
\label{prop:dmcontinuous}
Given $f\in\cts_G(\R^J)$ suppose the solution $(h,g)$ to the SP for $f$ satisfies the boundary jitter property (Definition \ref{def:jitter}). 
Let $\{f_k\}_{k\in\N}$ be a sequence of functions in $\cts_G(\R^J)$ such that $f_k$ converges to $f$ in $\cts_G(\R^J)$ as $k\to\infty$, and for each $k\in\N$ let $(h_k,g_k)$ denote the solution to the SP for $f_k$. 
Suppose $\psi\in\cts(\R^J)$ satisfies $\psi(0)\in \hyper_{h(0)}$ and $\{\psi_k\}_{k\in\N}$ is a sequence in $\cts(\R^J)$ converging to $\psi$ in $\cts(\R^J)$ as $k\to\infty$. 
Then $\dm_{h_k}^\alpha(\psi_k)$ converges to $\dm_h^\alpha(\psi)$ in $\dr(\R^J)$ as $k\to\infty$, where we recall that $\dr(\R^J)$ is equipped with the Skorokhod $J_1$-topology.
\end{prop}

\subsection{Derivative projection operators and their contraction properties}\label{sec:projection}

In this section we introduce and analyze so-called \emph{derivative projection operators}.
Derivative projection operators were introduced in \cite[Section 8]{Lipshutz2018} where they play an important role in establishing existence of directional derivatives of the SM when the constrained path reaches the nonsmooth part of the boundary $\U$.
In the next section the derivative projection operators are used to prove contraction properties for solutions of the derivative problem.

According to Assumption \ref{ass:setB}, there is a compact, convex, symmetric set $B^\alpha$ with $0\in (B^\alpha)^\circ$ satisfying \eqref{eq:setB}. A useful interpretation of $B^\alpha$ is in terms of an associated norm on $\R^J$ defined as follows:
	\be\label{eq:Bnorm}\norm{y}_{B^\alpha}:=\min\{r\geq0:y\in rB^\alpha\},\qquad y\in\R^J.\ee
The continuity condition in Assumption \ref{ass:setB} ensures that $\alpha\mapsto\norm{y}_{B^\alpha}$ is continuous for each $y\in\R^J$.
	
\begin{lem}[{\cite[Lemma 8.3]{Lipshutz2018}}]
\label{lem:projx}
For each $\alpha\in U$ and $x\in\partial G$ there exists a unique function
	$$\proj_x^\alpha:(\R^J,\norm{\cdot}_{B^\alpha})\mapsto(\R^J,\norm{\cdot}_{B^\alpha})$$
such that for each $y\in\R^J$,
	\be\label{eq:projxHx}\proj_x^\alpha y\in \hyper_x\qquad\text{and}\qquad\proj_x^\alpha y-y\in\spaan[d(\alpha,x)].\ee
Furthermore, $\proj_x^\alpha$ is linear and its operator norm, denoted $\norm{\proj_x^\alpha}$, satisfies
	$$\norm{\proj_x^\alpha}:=\sup_{y\neq 0}\frac{\norm{\proj_x^\alpha y}_{B^\alpha}}{\norm{y}_{B^\alpha}}\leq 1.$$
\end{lem}

\begin{remark}
	Let $\alpha\in U$ and $x\in G$. 
	From \eqref{eq:Hx}, \eqref{eq:dux} and \eqref{eq:projxHx} we see that $\hyper_x$ depends only on ${ \allN(x)}$ and $\proj_x^\alpha$ depends only on $\alpha\in U$ and $\allN(x)$.
\end{remark}

The next lemma states a contraction property of the derivative projection operators. The result plays a key role in proving contraction properties for solutions of the derivative problem and coupled solutions of the SP.

\begin{lem}\label{lem:delta}
	Let $U_0$ be a compact subset of $U$.
	There exists $\delta_0\in[0,1)$ such that if $\alpha\in U_0$, $K\in\N$ and $\{x_k\}_{k=1,\dots,K}$ is a finite sequence in $\partial G$ satisfying $\allN=\cup_{k=1,\dots,K}\allN(x_k)$, then for all $y\in\R^J$:
	\be\label{eq:projcontract}\norm{\proj_{x_K}^\alpha\cdots\proj_{x_1}^\alpha y}_{B^\alpha}\leq\delta_0\norm{y}_{B^\alpha}.\ee
\end{lem}

The proof of Lemma \ref{lem:delta}, which is given at the end of this subsection, relies on a related contraction property for  associated  adjoint operators.
In order to define the adjoint { operators}, let $B^{\alpha,\ast}$ denote the dual closed convex set given by
	$$B^{\alpha,\ast}:=\lcb y\in\R^J:\sup_{z\in B^\alpha}\ip{y,z}\leq 1\rcb.$$
Then $B^{\alpha,\ast}$ is a compact, convex, symmetric set with $0\in(B^{\alpha,\ast})^\circ$, and so, analogous to \eqref{eq:Bnorm}, $B^{\alpha,\ast}$ defines a norm $\norm{\cdot}_{B^{\alpha,\ast}}$ on $\R^J$ as follows:
	\be\label{eq:normBast}\norm{y}_{B^{\alpha,\ast}}:=\min\{r\geq0:y\in rB^{\alpha,\ast}\},\qquad y\in\R^J.\ee
	
\begin{remark}
Since $\alpha\mapsto B^\alpha$ is continuous in the Hausdorff metric by Assumption \ref{ass:setB}, it follows from the definition of $B^{\alpha,\ast}$ that $\alpha\mapsto B^{\alpha,\ast}$ is also continuous in the Hausdorff metric.
Therefore, the function $(\alpha,y)\mapsto\norm{y}_{B^{\alpha,\ast}}$ from $U\times\R^J$ to $\R_+$ is continuous.
\end{remark}

Let 
	$$\proj_x^{\alpha,\ast}:(\R^J,\norm{\cdot}_{B^{\alpha,\ast}})\mapsto(\R^J,\norm{\cdot}_{B^{\alpha,\ast}})$$
denote the adjoint operator of $\proj_x^\alpha$; that is, given $x\in G$, $\ip{\proj_x^\alpha y,z}=\ip{y,\proj_x^{\alpha,\ast} z}$ holds for all $y,z\in\R^J$.

\begin{lem}[{\cite[Lemmas 8.5 \& 8.6]{Lipshutz2018}}]
\label{lem:projxast}
For each $\alpha\in U$ and $x\in\partial G$
	$$\proj_x^{\alpha,\ast}:(\R^J,\norm{\cdot}_{B^{\alpha,\ast}})\mapsto(\R^J,\norm{\cdot}_{B^{\alpha,\ast}})$$
is the unique linear operator such that for each $y\in\R^J$,
	$$\proj_x^{\alpha,\ast} y\in\spaan[d(\alpha,x)]^\perp\qquad\text{and}\qquad\proj_x^{\alpha,\ast} y-y\in\hyper_x^\perp.$$
Furthermore, for each $y\in\R^J$,
	\begin{equation}
		\proj_x^{\alpha,\ast} y=y\qquad\text{if }y\in\spaan[d(\alpha,x)]^\perp,
	\end{equation}
	and
	\begin{equation}
		\norm{\proj_x^{\alpha,\ast} y}_{B^{\alpha,\ast}}<\norm{y}_{B^{\alpha,\ast}}\qquad\text{if }y\not\in\spaan[d(\alpha,x)]^\perp.
	\end{equation}

\end{lem}

In the following lemma we state a continuity property of the derivative projection operators and the adjoint operators.
Note that because $\proj_x^\alpha$ and $\proj_x^{\alpha,\ast}$ are finite-dimensional linear operators, they have representations as matrices in $\R^{J\times J}$.

\begin{lem}\label{lem:projcontinuous}
	For each $x\in\partial G$, the function $\alpha\mapsto\proj_x^\alpha$ from $U$ to $\R^{J\times J}$ is continuous, where $\R^{J\times J}$ is equipped with any fixed norm that does not depend on $\alpha$.
	Consequently, the function $\alpha\mapsto\proj_x^{\alpha,\ast}$ from $U$ to $\R^{J\times J}$ is also continuous.
\end{lem}

\begin{proof}
	To prove continuity of $\alpha\mapsto\proj_x^\alpha$ we use the explicit expression for $\proj_x^\alpha$ as a matrix obtained in \cite{Lipshutz2017a}.
	Fix $x\in\partial G$ and let $I:=\allN(x)$.
	For $\alpha\in U$ let $R_I(\alpha)$ and $N_I$ denote the $J\times\abs{I}$ matrices with column vectors $\{d_i(\alpha),i\in I\}$ and $\{n_i,i\in I\}$, respectively.
	By \cite[Lemma A.3]{Lipshutz2017a}, 
	$$\proj_x^\alpha=E_J-R_I(\alpha)(N_I^TR_I(\alpha))^{-1}N_I^T,$$
	where we recall that $E_J$ denotes the $J\times J$ identity matrix.
	The regularity of $\alpha\mapsto\proj_x^\alpha$ then follows from the fact that $d_i(\cdot)$, $i\in\allN$, and therefore, $R_I(\cdot)$, $I \subset \allN$, are continuous by assumption.
\end{proof}

\begin{lem}\label{lem:adjdelta}
	{For any  compact set} $U_0$ in $U$ there exists $\delta_0\in[0,1)$ such that given any $\alpha\in U_0$, $K\in\N$ and a finite sequence $\{x_k\}_{k=1,\dots,K}$ in $\partial G$ such that
		\be\label{eq:allNbigcup}\allN=\bigcup_{k=1,\dots,K}\allN(x_k),\ee
	then following inequality holds for all $y\in\R^J$:
		\be\label{eq:contract}\norm{\proj_{x_1}^{\alpha,\ast}\cdots\proj_{x_K}^{\alpha,\ast}y}_{B^{\alpha,\ast}}\le\delta_0\norm{y}_{B^{\alpha,\ast}}.\ee
\end{lem}

\begin{proof}
	Let ${\mathcal G}$ denote the set of sequences $\{x_k\}_{k=1,\dots,K}$ in $\partial G$, {for some $K \in \N$}, such that $\cup_{k=1,\dots,K}\allN(x_k)=\allN$. 
	Fix a sequence
		\be\label{eq:sequence}\lcb(\alpha_j,\{x_{j,k}\}_{k=1,\dots,K_j},y_j)\rcb_{j\in\N}\;\text{in }U_0\times  {\mathcal G} \times\sphere.\ee
{ Since $U_0 \times \sphere$ is compact}, by taking a subsequence if necessary, we can assume there exists $\alpha_0\in U_0$ and $y_0\in\sphere$ such that $(\alpha_j,y_j)\to(\alpha_0,y_0)$ as $j\to\infty$.
	By the definition of ${\mathcal G}$ and the linear independence of the directions of reflection stated in Assumption \ref{ass:independent}, for each $j\in\N$ there exists $1\le l_j\le K_j$ such that $y_0\not\in\spaan[d(\alpha_0,{ x_{j,l_j}})]^\perp$ and $y_0\in\spaan[d(\alpha_0,{ x_{j,k}})]^\perp$ for all $l_j<k\le K_j$. 
        { Furthermore, since there are only finitely many subsets of $\allN$,
          by choosing a further subsequence if necessary, we can assume that there exists a
          fixed subset $I \subseteq {\mathcal I}$
          such that $I(x_{j,l_j}) = I$ for all $j \in \N$.}  Since 
          the adjoint operator $\proj_x^{\alpha,\ast}$ { and $d(\alpha,x)$ depend} only on $\alpha$ and $\allN(x)$, 
          it follows that { $\proj_{x_{j,l_j}}^{\alpha_j,\ast}=\proj_{\bar{x}}^{\alpha_j,\ast}$
           and $d(\alpha_0,x_{j,l_j}) = d (\alpha_0, \bar{x})$  for 
          any fixed $\bar{x} \in \partial G$ with $I(\bar{x}) = I$. } 
	Then by Lemma \ref{lem:projxast}, we see that
		$$\norm{\proj_{x_{j,1}}^{\alpha_j,\ast}\cdots\proj_{x_{j,K_j}}^{\alpha_j,\ast}y_j}_{B^{\alpha_j,\ast}}\le\norm{\proj_{x_{j,l_j}}^{\alpha_j,\ast}\cdots\proj_{x_{j,K_j}}^{\alpha_j,\ast}y_j}_{B^{\alpha_j,\ast}}=\norm{\proj_{{ \bar{x}}}^{\alpha_j,\ast}z_j}_{B^{\alpha_j,\ast}},$$ 
	where $z_j:=\proj_{x_{j,l_j+1}}^{\alpha_j,\ast}\cdots\proj_{x_{j,K_j}}^{\alpha_j,\ast}y_j$. 
	We claim, and prove below, that $z_j\to y_0$ as $j\to\infty$.
	Assuming the claim holds, we have for each $j\in\N$,
	\begin{align*}
		\norm{\proj_{{ \bar{x}}}^{\alpha_j,\ast}z_j}_{B^{\alpha_j,\ast}}&\le\left|\norm{{ \proj_{\bar{x}}}^{\alpha_j,\ast}z_j}_{B^{\alpha_j,\ast}}-\norm{\proj_{{ \bar{x}}}^{\alpha_0,\ast}y_0}_{B^{\alpha_0,\ast}}\right|+\norm{\proj_{{ \bar{x}}}^{{ \alpha_0},\ast}y_0}_{B^{\alpha_0,\ast}}.
	\end{align*}
	The continuity of the adjoint operators shown in Lemma \ref{lem:projcontinuous} , the convergence of $(\alpha_j,z_j)\to(\alpha_0,y_0)$ as $j\to\infty$, and the continuity of $(\alpha,y)\mapsto\norm{y}_{B^{\alpha,\ast}}$ imply that the first term on the right hand side converges to zero as $j\to\infty$.
	For the second term, recall that $y_0\not\in\spaan[d(\alpha_0,{ \bar{x}})]^\perp$ and so $\norm{\proj_{{ \bar{x}}}^{\alpha_0,\ast}y_0}_{B^{\alpha_0,\ast}}<\norm{y_0}_{B^{\alpha_0,\ast}}$ by Lemma \ref{lem:projxast}.
	Combining the above yields,
		$$\limsup_{j\to\infty}\norm{\proj_{x_{j,1}}^{\alpha_j,\ast}\cdots\proj_{x_{j,K_j}}^{\alpha_j,\ast}y_j}_{B^{\alpha_j,\ast}}\le\norm{\proj_{{ \bar{x}}}^{\alpha_0,\ast}y_0}_{B^{\alpha_0,\ast}}<\norm{y_0}_{B^{\alpha_0,\ast}}=\lim_{j\to\infty}\norm{y_j}_{B^{\alpha_j,\ast}},$$
	where the final equality follows from the convergence of $(\alpha_j,z_j)\to(\alpha_0,y_0)$ as $j\to\infty$ and the continuity of $(\alpha,y)\mapsto\norm{y}_{B^{\alpha,\ast}}$.
	Since this holds for every  such sequence \eqref{eq:sequence},
        it follows that there exists $\delta_0\in[0,1)$ such that \eqref{eq:contract} holds for all $y\in\sphere$.
	The lemma then follows from the homogeneity property of the norm $\norm{\cdot}_{B^{\alpha,\ast}}$.
	
	We are left to prove the claim that $z_j\to y_0$ as $j\to\infty$.
	Recall from our choice of $1\le l_j\le K_j$ that $y_0\in\spaan[d(\alpha,{ x_{j,k}})]^\perp$ for all $l_j<k\le K_j$.
	Thus, by Lemma \ref{lem:projxast}, $y_0=\proj_{x_{j,l_j+1}}^{\alpha,\ast}\cdots\proj_{x_{j,K_j}}^{\alpha,\ast}y_0$ for each $j\in\N$. 
	Along with the definition of $z_j$ and the nonexpansive property of the adjoint operators, this implies that
        \[ { \norm{z_j-y_0}_{B^{\alpha_j,\ast}}\le \norm{\proj_{x_{j,l_j+1}}^{\alpha_j,\ast}\cdots\proj_{x_{j,K_j}}^{\alpha_j,\ast}(y_j-y_0)} \le   \norm{y_j-y_0}_{B^{\alpha_j,\ast}}.} 
        \]
	The claim then follows from the convergence $(\alpha_j,y_j)\to(\alpha_0,y_0)$ as $j\to\infty$ and the continuity of $(\alpha,y)\mapsto\norm{y}_{B^{\alpha,\ast}}$.
\end{proof}

\begin{proof}[Proof of Lemma \ref{lem:delta}]
	Fix a compact set $U_0$ in $U$ and let $\delta_0\in[0,1)$ be as in Lemma \ref{lem:adjdelta}.
By Lemma \ref{lem:delta}, for all $y\in\R^J$,
	\be\label{eq:adjointcontract}\norm{\proj_{x_1}^{\alpha,\ast}\cdots\proj_{x_K}^{\alpha,\ast} y}_{B^{\alpha,\ast}}\leq\delta_0\norm{y}_{B^{\alpha,\ast}}.\ee
Then by the definition of the set $B^{\alpha,\ast}$ norm in \eqref{eq:normBast}, we see that $\proj_{x_1}^{\alpha,\ast}\cdots\proj_{x_K}^{\alpha,\ast} z\in \delta_0 B^{\alpha,\ast}:=\{\delta_0 y:y\in B^{\alpha,\ast}\}$ for all $z\in\R^J$. Thus, given $y\in\R^J$,
\begin{align*}
	\norm{\proj_{x_K}^\alpha\cdots\proj_{x_1}^\alpha y}_{B^\alpha}&=\sup_{z\in\partial B^{\alpha,\ast}}\ip{\proj_{x_K}^\alpha\cdots\proj_{x_1}^\alpha y,z}\leq\sup_{z\in\partial(\delta(\alpha) B^{\alpha,\ast})}\ip{y,z}=\delta_0\norm{y}_{B^\alpha},
\end{align*}
which proves \eqref{eq:projcontract}.
\end{proof}

\subsection{Contractions of solutions to the derivative problem}\label{sec:dpstable}

In this section we a prove a contraction property for solutions of the derivative problem along a path that ``visits'' every face in a finite time interval.
The following is the main result of this section.

\begin{prop}\label{prop:dpcontract}
	Given a compact subset $U_0$ of $U$, let $\delta_0\in[0,1)$ be as in Lemma \ref{lem:delta}.
	Let $\alpha\in U$, $\x\in\cts_G(\R^J)$ and $\psi\in\cts(\R^J)$, and suppose $(\z,\y)$ is the solution to the SP $\{(d_i(\alpha),n_i),i\in\allN\}$ for $\x$ and $(\phi,\eta)$ is the solution to the derivative problem along $\z$ for $\psi$. 
	Then for all $0\leq S<T<\infty$.
	\be\label{eq:dmcontract}\norm{\phi(T)}_{B^\alpha}\leq\delta_0^{1{\{\cup_{t\in[S,T]}\allN(\z(t))=\allN\}}}\norm{\phi(S)}_{B^\alpha}+\lip_\dm(\alpha)\sup_{S\leq u\leq T}\norm{\psi(u)-\psi(S)}_{B^\alpha}.\ee
\end{prop}

The remainder of this section is devoted to the proof of Proposition \ref{prop:dpcontract}. 
Throughout this section we fix a compact subset $U_0$ of $U$ and let $\delta_0\in[0,1)$ be as in Lemma \ref{lem:delta}.  
In addition, we fix $\alpha\in U$, $\x\in\cts_G(\R^J)$ and $\psi\in\cts(\R^J)$, and let $(\z,\y)$ denote the solution to the SP $\{(d_i(\alpha),n_i),i\in\allN\}$ for $\x$ and $(\phi,\eta)$ denote the solution to the derivative problem along $\z$ for $\psi$. 
In addition, fix $0\le S<T<\infty$.    { We start with some preparatory lemmas. } 

\begin{lem}\label{lem:psiconstrho}
	Suppose $\allN(\z(u))\subseteq\allN(\z(S))$ for all $u\in[S,T]$ and $\psi$ is constant on $[S,T]$. 
	Then $(\phi,\eta)$ is constant on $[S,T]$.
\end{lem}

\begin{proof}
	Define $\z^S$, $\phi^S$, $\eta^S$ and $\psi^S$ as in \eqref{eq:zS}, \eqref{eq:phis}--\eqref{eq:psis}, respectively. 
	By \eqref{eq:phis} and \eqref{eq:etas} we need to show that $(\phi^S,\eta^S)$ is constant on $[0,T-S]$. 
	According to the time-shift property of the derivative problem (Lemma \ref{lem:dptimeshift}), $(\phi^S,\eta^S)$ is a solution to the derivative problem along $\z^S$ for $\psi^S$. 
	Define $\wt\phi^S$ and $\wt\eta^S$ by $\wt\phi^S({ u}):=\phi(s)$ and $\wt\eta^S(u):=\eta(s)$ for all $u\geq0$. 
	We claim, and prove below,  that $(\wt\phi^S,\wt\eta^S)$ is a solution to the derivative problem along $\z^S$ for $\psi^S$ on the interval $[0,T-S]$. 
	It then follows from the uniqueness of solutions to the derivative problem that $(\wt\phi^S,\wt\eta^S)=(\phi^S,\eta^S)$, thus completing the proof.

	We are left to show that $(\wt\phi^S,\wt\eta^S)$ satisfies conditions 1--3 of the derivative problem along $\z^S$ for $\psi^S$ on the interval $[0,T-S]$. By the definition of $\wt\phi^S$, condition 1 of the derivative problem, \eqref{eq:psis} and the definition of $\wt\eta^S$, we have 
	$$\wt\phi^S(u)=\phi(s)=\psi(s)+\eta(s)=\psi^S(u)+\wt\eta^S(u),\qquad u\in[0,T-S].$$ 
	Thus $(\wt\phi^S,\wt\eta^S)$ satisfies condition 1 of the derivative problem on $[0,T-S]$. Next, given $u\in[0,T-S]$, observe that $\allN(\z^S(u))=\allN(\z(s+u))\subseteq\allN(\z(s))$ for all $u\in[0,T-S]$ and the definition of $\hyper_x$ in \eqref{eq:Hx} imply that $\hyper_{\z(S)}\subseteq \hyper_{\z^S(u)}$. It follows from condition 2 of the derivative problem that for all $u\in[0,T-S]$,
	$$\wt\phi^S(u)=\phi(s)\in \hyper_{\z(s)}\subseteq \hyper_{\z^S(u)}.$$
	This proves that $\wt\phi^S$ satisfies condition 2 of the derivative problem for all $u\in[0,T-S]$. Lastly, since $\wt\eta^S$ is constant on $[0,T-S]$, $\wt\eta^S$ automatically satisfies condition 3 of the derivative problem on the interval $[0,T-S]$.
\end{proof}

Set $t_1:= S$ and for $k\geq 1$ such that $t_k<T$, recursively define
	\be\label{eq:rhoksequence}\rho_k:=\inf\{t>t_k:\allN(\z(t))\not\subseteq\allN(\z(t_k))\}\wedge T,\ee
to be the first time the constrained path hits a face not contained in $\allN(h(t))$, 
 and
	\be\label{eq:tksequence}t_{k+1}:=\sup\{t\geq\rho_k:\allN(\z(s))\subseteq\allN(\z(t))\;\forall\;s\in(t_k,t]\}\wedge T,\ee
noting that it is possible for $t_{k+1} = \rho_k$.
We claim that $t_K=T$ for some $K\in\N$ and $t_k<T$ for all $1\leq k<K$. For an argument by contradiction, suppose that $t_k<T$ for all $k\in\N$. Then there exists $t_\infty\leq T$ such that $t_k\to t_\infty$ as $k\to\infty$. 
Due to the continuity of $\z$ and the upper semicontinuity of $\allN(\cdot)$ (see, e.g., \cite[Lemma 2.1]{Kang2007}), this implies there exists $k_0\in\N$ such that $\allN(\z(t_k))\subseteq\allN(\z(t_\infty))$ for all $k>k_0$. However, \eqref{eq:tksequence} then implies that $t_{k_0+1}=t_\infty$, which is a contradiction. Thus, the claim holds.

\begin{lem}
	\label{lem:phiTpsiconst}
	Suppose $\psi\in\cts(\R^J)$ is constant on $[S,T]$. Then
		\be\label{eq:phiprojtKt1phi0}\phi(T)=\proj_{\z(t_K)}^\alpha\cdots\proj_{\z(t_1)}^\alpha\phi(S).\ee
\end{lem}

\begin{proof}
	Since $t_K=T$, in order to prove \eqref{eq:phiprojtKt1phi0} it suffices to show that 
	\be\phi(t_k)=\proj_{\z(t_k)}^\alpha\phi(t_{k-1}),\qquad k=2,\dots,K.\ee
	Let $2\leq k\leq K$. By \eqref{eq:rhoksequence}, $\allN(\z(t))\subseteq\allN(t_{k-1})$ for all $t\in[t_{k-1},\rho_{k-1})$. Therefore, by Lemma \ref{lem:psiconstrho} and the fact that $\psi$ is constant on $[S,T]$, $(\phi,\eta)$ is constant on $[t_{k-1},\rho_{k-1})$. By \eqref{eq:tksequence},
	\be\label{eq:rhok1tk}\allN(\z(t))\subseteq\allN(\z(t_k))\qquad\text{for all }t\in[\rho_{k-1},t_k].\ee
	Due to the continuity of $\z$ and the upper semicontinuity of $\allN(\cdot)$ there exists $s_{k-1}<\rho_{k-1}$ such that $\allN(\z(t))\subseteq\allN(\z(\rho_{k-1}))$ for all $t\in[s_{k-1},\rho_{k-1}]$, which along with \eqref{eq:rhok1tk} implies
	\be\label{eq:sk1tk}\allN(\z(t))\subseteq\allN(\z(t_k))\qquad\text{for all }t\in[s_{k-1},t_k].\ee
	Now by condition 2 of the derivative problem, we have $\phi(t_k)\in \hyper_{\z(t_k)}$. Condition 1 of the derivative problem, the fact that $\psi$ is constant on $[S,T]$, condition 3 of the derivative problem and \eqref{eq:sk1tk} together yield
	$$\phi(t_k)-\phi(s_{k-1})=\eta(t_k)-\eta(s_{k-1})\in\spaan\lsb\cup_{u\in(s_{k-1},t_k]}d(\alpha,\z(u))\rsb\subseteq\allN(\z(t_k)).$$
	Thus, due to the uniqueness of the derivative projection operators stated in Lemma \ref{lem:projx} and the fact that $\phi$ is constant on $[t_{k-1},\rho_{k-1})$, we see that
	$$\phi(t_k)=\proj_{\z(t_k)}^\alpha\phi(s_{k-1})=\proj_{\z(t_k)}^\alpha\phi(t_{k-1}),$$
	which proves \eqref{eq:rhok1tk}.
\end{proof}

\begin{proof}[Proof of Proposition \ref{prop:dpcontract}]
	Define $K\in\N$ and $\{t_k\}_{k=1,\dots,K}$ as in \eqref{eq:rhoksequence}--\eqref{eq:tksequence}. 
	Define $\psi_1,\psi_2\in\cts(\R^J)$, for $t\geq0$, by 
	\begin{align}\label{eq:psi1}
	\psi_1(t)&:=\phi(S),\\ \label{eq:psi2}
	\psi_2(t)&:=\psi(S+t)-\psi(S).
	\end{align}
	Define $\z^S\in\cts(G)$, $\phi^S\in\cts(\R^J)$ and $\psi^S\in\dr(\R^J)$ as in \eqref{eq:zS}, \eqref{eq:phis} and \eqref{eq:psis}, respectively. 
	Then $\psi^S=\psi_1+\psi_2$ and by the time-shift property of the derivative problem shown in Lemma \ref{lem:dptimeshift}, $\phi^S=\dm_{\z^S}^\alpha(\psi^S)$. 
	Let $\phi_1:=\dm_{\z^S}^\alpha(\psi_1)$ and $\phi_2:=\dm_{\z^S}^\alpha(\psi_2)$. 
	The linearity of the derivative map shown in Lemma \ref{lem:dmlinear} implies that
		\be\label{eq:phisphi1phi2}\phi^S=\phi_1+\phi_2.\ee
	Since $\psi_1$ is a constant function, by  Lemma \ref{lem:phiTpsiconst}, we have 
		\be\label{eq:normphi1B}\norm{\phi_1(T-S)}_{B^\alpha}=\norm{\proj_{\z(t_K)}^\alpha\cdots\proj_{\z(t_1)}^\alpha\phi_1(0)}_{B^\alpha}=\norm{\proj_{\z(t_K)}^\alpha\cdots\proj_{\z(t_1)}^\alpha\phi(S)}_{B^\alpha}.\ee 
	By the Lipschitz continuity of the derivative map and the definition of $\psi_2$ in \eqref{eq:psi2},
	\begin{align}\label{eq:phi2bound}
	\norm{\phi_2(T-S)}_{B^\alpha}&\leq\lip_\dm(\alpha)\sup_{0\leq u\leq T-S}\norm{\psi_2(u)}_{B^\alpha}\leq\lip_\dm(\alpha)\sup_{S\leq u\leq T}\norm{\psi(u)-\psi(S)}_{B^\alpha}.
	\end{align}
	The definition of $\phi^S$ in \eqref{eq:phis} along with \eqref{eq:phisphi1phi2}--\eqref{eq:phi2bound} yields 
		\be\label{eq:phitBboundproj}\norm{\phi(T)}_{B^\alpha}\leq\norm{\proj_{\z(t_K)}^\alpha\cdots\proj_{\z(t_1)}^\alpha\phi(S)}_{B^\alpha}+\lip_\dm(\alpha)\sup_{S\leq u\leq T}\norm{\psi(u)-\psi(S)}_{B^\alpha}.\ee
	
	Suppose
		\be\label{eq:cupSTallN}\bigcup_{t\in[S,T]}\allN(\z(t))=\allN.\ee
	Define the sequences $\{t_k\}_{k=1,\dots,K}$ and $\{\rho_k\}_{k=1,\dots,K}$ as in \eqref{eq:rhoksequence}--\eqref{eq:tksequence}. 
	We show that $\cup_{k=1,\dots,K}\allN(\z(t_k))=\allN$. 
	Set $\rho_0:= S$, $\rho_K:= T$ and observe that \eqref{eq:rhoksequence} and \eqref{eq:tksequence} together imply that for each $1\leq k< K$,
		$$\bigcup_{t\in[\rho_{k-1},\rho_k)}\allN(\z(t))\subseteq\allN(\z(t_k)).$$
	Along with the assumption \eqref{eq:cupSTallN} and the fact that $t_K=\rho_K=T$, this implies
		$$\bigcup_{k=1,\dots,K}\allN(\z(t_k))=\bigcup_{t\in[S,T]}\allN(\z(t))=\allN.$$
	Then \eqref{eq:phitBboundproj}, the contraction property of the derivative projection operators (Lemma \ref{lem:delta}) and the fact that the operator norms of the derivative projection operators are bounded by one (Lemma \ref{lem:projx}) together prove \eqref{eq:dmcontract}.
\end{proof}

\section{Feller Markov property of the joint process}\label{sec:Feller}

In this section we prove Theorem \ref{thm:jointmarkov}. 
The proof is given at the end of the section after we establish some preliminary lemmas. 
Throughout this section we assume the data $\{(d_i(\cdot),n_i),i\in\allN\}$ satisfies Assumptions \ref{ass:independent}, \ref{ass:setB} and \ref{ass:projection}, and the coefficients $b(\cdot)$, $\sigma(\cdot)$ and $R(\cdot)$ satisfy Assumption \ref{ass:holder}, and we fix $\alpha\in U$. 
Recall the definition of $\state$ given in \eqref{eq:state}.
In addition, for $\xi=(x,y)\in\state$, recall that $\Z^x$ denotes the RBM starting at $x$, $\phib^\xi$ denotes the derivative process along $\Z^x$ starting at $y$ and $\Xi^\xi=(\Z^x,\phib^\xi)$ denotes the joint process starting at $\xi$ and taking values in $\state$.
We begin by defining a measurable map 
	$$\Pi:\dom(\Pi)\mapsto\cts(G)\times\dr(\R^J)$$ 
such that $\dom(\Pi)\subset\state\times\cts_0(\R^J)$ and for each $\xi\in\state$, almost surely $(\xi,\bm)\in\dom(\Pi)$, $\Pi$ is continuous at $(\xi,\bm)$ and $\Pi(\xi,\bm)=\Xib^\xi$. 
With this in mind, given $\xi=(x,y)\in\state$ and $w\in\cts_0(\R^J)$, define $\x\in\cts_G(\R^J)$ by
	\be\label{eq:xdef}\x(t):= x+bt+\sigma w(t),\qquad t\geq0,\ee
let $(\z,\y)$ denote the solution to the SP for $\x$, and define $\ell\in\cts(\R^J)$ by $\ell(\cdot):= R^{-1}\y(\cdot)$. In addition, define $\psi\in\cts(\R^{J})$ by
	\be\label{eq:Psidef}\psi(t):= y+b't+\sigma'w(t)+R'\ell(t),\qquad t\geq0.\ee
By the continuity of the SM (Proposition \ref{prop:SP}) and the definition of $\ell$, the mapping $(\xi,w)\mapsto\ell$ from $\state\times\cts_0(\R^J)$ to $\cts(\R^J)$ is continuous. 
Thus, by \eqref{eq:xdef} and \eqref{eq:Psidef}, the mapping $(\xi,w)\mapsto(\x,\psi)$ from $\state\times\cts_0(\R^J)$ to $\cts(\R^J)\times\cts(\R^{J})$ is continuous. 
Define
	\be\label{eq:domPi}\dom(\Pi):=\lcb(\xi,w)\in\state\times\cts_0(\R^J):(\z,\y)\text{ satisfies the boundary jitter property}\rcb.\ee
By Proposition \ref{prop:dpexistence}, if $(\xi,w)\in\dom(\Pi)$ then the derivative map $\dm_\z$ is well defined.
In this case we define the function $\phi\in\dr(\R^{J})$ by $\phi:=\dm_\z(\psi)$ and set 
	\be\label{eq:Pi}\Pi(\xi,w):=(\z,\phi),\qquad\Pi_t(\xi,w):=(\z(t),\phi(t)),\; t\geq0.\ee
By the definition of $\Pi$, the continuity of $(\xi,w)\mapsto(\x,\psi)$ shown above, the continuity of the SM and the continuity property of the derivative map shown in Proposition \ref{prop:dmcontinuous}, $\Pi$ is continuous on $\dom(\Pi)$.
	
In the following lemmas we prove a time-shift (or semi-group) property for $\Pi$ and express the joint process in terms of $\Pi$. 

\begin{lem}\label{lem:Pitimeshift}
	For $(\xi,w)\in\dom(\Pi)$ and $s,t\ge0$ such that $\z(s)\in G^\circ$, it holds that $(\Pi_s(\xi,w),\Theta_sw)\in\dom(\Pi)$ and
	\be \Pi_{t+s}(\xi,w)=\Pi_t(\Pi_s(\xi,w),\Theta_sw).\ee
\end{lem}

\begin{proof}
	Define $\z^S$, $\y^S$, $\x^S$, $\phi^S$ and $\psi^S$ as in \eqref{eq:zS}, \eqref{eq:yS}, \eqref{eq:xS}, \eqref{eq:phis} and \eqref{eq:psis}, respectively.
	The fact that $(\z,\y)$ satisfies the boundary jitter property and $\z(s)\in G^\circ$ implies that $(\z^S,\y^S)$ satisfies the boundary jitter property and so $(\Pi_s(\xi,w),\Theta_sw)\in\dom(\Pi)$.
	By \eqref{eq:Pi}, \eqref{eq:zS}, \eqref{eq:phis}, \eqref{eq:xS}, \eqref{eq:psis}, the time-shift properties of the SP (Lemma \ref{lem:SPtimeshift}) and the derivative problem (Lemma \ref{lem:dptimeshift}),	
	\begin{align*}
		\Pi_{t+s}(\xi,w)&=(\z^S(t),\phi^S(t))=(\sm(\x^S),\dm_{\z^S}(\psi^S))=\Pi_t(\Pi_s(\xi,w),\Theta_sw).
	\end{align*}
\end{proof}

\begin{lem}\label{lem:Pi}
For each $\xi\in\state$, almost surely the following hold: $(\xi,\bm)\in\dom(\Pi)$ and $\Xib^\xi=\Pi(\xi,\bm)$.
Consequently, $\Pi$ is continuous at $(\xi,\bm)$.
\end{lem}

\begin{proof}
Let $\xi=(x,y)\in\state$ and $\Xib^\xi=(\Z^x,\phib^\xi)$ denote the joint process starting at $\xi$.
Let $\Y^x$ be as in Definition \ref{def:rbm} and $\etab^\xi$ be as in Definition \ref{def:de}. 
By Remark \ref{rmk:ZsmX}, a.s.\ $(\Z^x,\Y^x)$ is the solution to the SP for $\X^x(\cdot):=x+b\iota(\cdot)+\sigma\bm(\cdot)$; that is, a.s.\ $\Z^x=\sm(\X^x)$. 
Furthermore, by Proposition \ref{prop:jitter}, a.s.\ $(\Z^x,\Y^x)$ satisfies the boundary jitter property.
Thus, from \eqref{eq:domPi} we see that a.s.\ $(\xi,\bm)\in\dom(\Pi)$.
Moreover, by Remark \ref{rmk:derivativeprocessbeta}, a.s.\ $(\phib^\xi,\etab^\xi)$ is the solution to the derivative problem along $\Z^x$ for $\psib^\xi$, where $\psib^\xi(\cdot):=y+b'\iota(\cdot)+\sigma'\bm(\cdot)+R'\L^x(\cdot)$ and $\L^x(\cdot)=R^{-1}\Y^x(\cdot)$.
In other words, a.s.\ $\phib^\xi=\dm_{\Z^x}(\psib^\xi)=\dm_{\sm(\X^x)}(\psib^\xi)$.
Thus, by the construction of $\Pi$, a.s.\ $\Xib^\xi=\Pi(\xi,\bm)$.
The continuity of $\Pi$ at $(\xi,\bm)$ then follows from the fact that $\Pi$ is continuous on $\dom(\Pi)$. 
\end{proof}

\begin{lem}\label{lem:Pit}
For each $s,t>0$ and $\xi\in\state$, almost surely the following hold: $(\Xi(s),\Theta_s\bm)\in\dom(\Pi)$, ${\color{blue} \Pi_t}$ is continuous at $(\xi,\bm)$ and
	\be\label{eq:Pitcts}\Xib(s+t)=\Pi_t(\Xib(s),\Theta_s\bm).\ee
\end{lem}

\begin{proof}
	Let $s,t>0$ and $\xi=(x,y)\in\state$.
	Recall that $\Z^x$ is the RBM starting at $x$. 
	Then a.s.\ $\Z^x(\cdot)$ is continuous and $\Z^x(s)\in G^\circ$ (see, e.g., \cite[Lemma 5.7]{Budhiraja2007}). 
	The fact that \eqref{eq:Pitcts} holds follows from Lemma \ref{lem:Pitimeshift}.
	Next, we prove that a.s.\ $\Pi_t$ is continuous at $(\xi,\bm)$. 
	Since a.s.\ $\Pi$ is continuous at $(\xi,\bm)$ by Lemma \ref{lem:Pi},  it suffices to show that a.s.\ $\Pi(\xi,\bm)(\cdot)$ is continuous at $t$. We have a.s.\ $\Pi(\xi,\bm)=\Xib^\xi=(\Z^x,\phib^\A)$. Since $\Z^x(\cdot)$ is continuous and $\phib^\A(\cdot)$ is continuous at $t$ if $\Z^x(t)\in G^\circ$ (part (ii) of Theorem \ref{thm:pathwise}), it follows that a.s.\ $\Pi_t$ is continuous at $(\xi,\bm)$.
\end{proof}

We can now prove Theorem \ref{thm:jointmarkov}.

\begin{proof}[Proof of Theorem \ref{thm:jointmarkov}]
Let $s,t>0$. By Lemma \ref{lem:Pit} and the facts that $\Xib(s)$ is $\F_s$-measurable and $\Theta_s\bm$ is independent of $\F_s$, we have, for every bounded measurable function $\zeta:\state\mapsto\R$,
\begin{align*}
	\E[\zeta(\Xib(s+t))|\F_s]&=\E[(\zeta\circ \Pi_t)(\Xib(s),\Theta_s\bm)|\F_s]\\
	&=\E[(\zeta\circ\Pi_t)(\Xib(s),\Theta_s\bm)|\Xib(s)]\\
	&=\E[\zeta(\Xib(s+t))|\Xib(s)],
\end{align*}
which shows that $\{\Xib(t),\F_t,t\geq0\}$ satisfies the Markov property.
 
To see that Feller continuity holds, recall that $\{P_t\}$ denotes the family of transition functions on $\state$ defined in \eqref{eq:Pt} and let $\zeta:\state\mapsto\R$ be a bounded continuous function. By Lemma \ref{lem:Pit}, a.s.\ $\Pi_t(\cdot,\bm)$ is continuous at $\xi$. Thus, given a sequence $\{\xi_n\}_{n\in\N}$ in $\state$ converging to $\xi$, it follows from bounded convergence that
\begin{align*}
	\lim_{n\to\infty}(P_t \zeta)(\xi_n)&=\lim_{n\to\infty}\E[(\zeta\circ \Pi_t)(\xi_n,\bm)]\\
	&=\E[(\zeta\circ\Pi_t)(\xi,\bm)]\\
	&=(P_t \zeta)(\xi).
\end{align*}
Combining the above, we see that the family of transition functions $\{P_t\}$ is Markovian and Feller continuous.
\end{proof}

\section{Estimates for the probability an RBM visits every face}
\label{sec:rbmestimates}

In this section we obtain lower bounds on the probability that an RBM visits every face of the polyhedral cone $G$ in a compact time interval, which imply that the RBM almost surely repeatedly visits every face of $G$.
These results are used to prove stability of the joint process in Section \ref{sec:stable} and uniqueness of the stationary distribution for the joint process in Section \ref{sec:jointunique}.
Throughout this section we assume the data $\{(d_i(\cdot),n_i),i\in\allN\}$ satisfies Assumptions \ref{ass:independent}, \ref{ass:setB} and \ref{ass:projection}, and the coefficients $b(\cdot)$, $\sigma(\cdot)$ and $R(\cdot)$ satisfy Assumptions \ref{ass:holder} and \ref{ass:stable}. 

\subsection{Lower bounds for the probability an RBM visits every face in an interval}
\label{sec:lowerbound}

Define $G_\infty := G$  and for $K\in(0,\infty)$ define
	$$G_K:=\{x\in G:\abs{x}\le K\}.$$ 
Also, for $K \in (0,\infty]$, $\alpha\in U$ and $x\in G$ define the sequence of $\{\F_t\}$-stopping times $\{\tau_K^{\alpha,x}(j)\}_{j\in\N}$ as follows: 
set $\tau_K^{\alpha,x}(0):=0$ and for $j\in\N$ such that $\tau_K^{\alpha,x}(j-1)<\infty$ recursively define
\be
	\label{eq:tauj}
	\tau_K^{\alpha,x}(j):=\inf\lcb t>\tau_K^{\alpha,x}(j-1):\bigcup_{s\in[\tau_K^{\alpha,x}(j-1),t]}\allN(\Z^{\alpha,x}(s))=\allN,\;\Z^{\alpha,x}(t)\in G_K\rcb.
\ee
If $\tau_K^{\alpha,x}(j)=\infty$ for some $j\in\N$, then set $\tau_K^{\alpha,x}(k):=\infty$ for all $k>j$.
When $K=\infty$ we drop the subscript $K$ and write $\tau^{\alpha,x}(j)$ for $\tau_\infty^{\alpha,x}(j)$.

The following is the main result of this section.

\begin{prop}\label{prop:taujfinite}
	Given a compact subset $U_0$ of $U$ and $K,T\in(0,\infty)$, there exists $\gamma_0\in(0,1)$ such that for all $\alpha\in U_0$ and $x\in G_K$,
	\be\label{eq:tau1T}\P(\tau_K^{\alpha,x}(1)\le T)\ge\gamma_0.\ee
	Consequently, for each $\alpha\in U$ and $x\in G$, almost surely $\tau_K^{\alpha,x}(j)<\infty$ for all $j\in\N$.
\end{prop}

\begin{remark}\label{rmk:tau}
Let $\alpha\in U$ and $x\in G$.
Since $\tau^{\alpha,x}(j)\le\tau_K^{\alpha,x}(j)$ for all $K\in(0,\infty)$ and $j\in\N$ by definition, it follows that almost surely $\tau^{\alpha,x}(j)<\infty$ for all $j\in\N$. 
\end{remark}

The remaining subsections are devoted to the proof of Proposition \ref{prop:taujfinite}.
Throughout these subsections we fix a compact subset $U_0$ in $U$.

\subsection{Lower bounds for the probability an RBM visits a single face in an interval}
\label{sec:RBMestimates}

The following lemma provides a uniform lower bound on the probability that an RBM hits the $i$th face (and does not hit any other face) in a finite time interval.
For $K\in(0,\infty)$ and $\ve>0$, define
	$$G_K^\ve:=\{x\in G_K:\rho(x,\partial G)\geq\ve\},$$
where $\rho(\cdot,\cdot)$ denotes the Euclidean metric on $\R^J$ and $G_K$ is the set defined in Section \ref{sec:lowerbound}.

\begin{lem}\label{lem:Pcup0T1}
	Given $K,T\in(0,\infty)$ and $\ve\in(0,K/2)$, there exist $\wt K\ge K$ and $\wt\gamma_0\in(0,1)$ such that for all $\alpha\in U_0$, $x\in G_K^\ve$ and $i\in\allN$,
		\be\label{eq:wtgamma0}\P\lb\bigcup_{s\in[0,T]}\allN(\Z^{\alpha,x}(s))=\{i\},\;\Z^{\alpha,x}(T)\in G_{\wt K}^\ve\rb\geq\wt\gamma_0.\ee
\end{lem}

Fix $K,T\in(0,\infty)$ and $\ve\in(0,K/2)$.
Given $\alpha_0\in U_0$ we construct, for each $i\in\allN$, a continuous path $w_i$ on $[0,T]$ (for some $T<\infty$) taking values in $\R^J$ such that if the driving Brownian motion $\bm$ remains within a certain specified neighborhood of $w_i$ on $[0,T]$, then for any $\alpha$ in a neighborhood $U_{\alpha_0}$ of $\alpha_0$ and $x\in G_K^\ve$, the RBM $\Z^{\alpha,x}$ will hit face $F_i$ in the interval $[0,T]$, will not hit any other face in the interval, and will lie in the set $G_{\wt K}^\ve$ at the end of the interval, for some $\wt K\ge K$, which may depend on $\alpha_0$.
Since $\bm$ remains in the specified neighborhood of $w_i$ on $[0,T]$ with positive probability and $U_0$ is compact, this will complete the proof of the lemma.

Before proving Lemma \ref{lem:Pcup0T1}, we define the path $w_i$ and summarize relevant properties.
Since the vectors $\{n_j,j\in\allN\}$ are linearly independent by Assumption \ref{ass:independent}, for each $i\in\allN$ there is a unique vector $\wt n_i\in\partial G$ such that 
		\be\label{eq:wtni}\ip{\wt n_i,n_i}=1,\qquad\ip{\wt n_i,n_j}=0\quad\forall\; j\neq i.\ee
	Define $u\in G^\circ$ by $u:=\sum_{i=1}^J\wt n_i$ so that
		\be\label{eq:u}\ip{u,n_i}=1,\qquad i\in\allN.\ee 
	For each $i\in\allN$ define $v_i\in\R^J$ by
		\be\label{eq:vi}v_i:=-\wt n_i+\sum_{j\neq i}(\ip{d_i(\alpha_0),n_j})^-\wt n_j,\ee
	and the continuous function $w_i:[0,T]\mapsto\R^J$ by
		\be\label{eq:wi}
	w_i(t):=(\sigma(\alpha_0))^{-1}\lb \bbr_1v_i\lb t\wedge\frac{T}{2}\rb+\bbr_2u\lb t-\frac{T}{2}\rb^+-b(\alpha_0)t\rb,\qquad t\in[0,T],
	\ee
	where
		\be\label{eq:r1r2}\bbr_1:=\frac{2(K+\ve)}{T},\qquad \bbr_2:= \frac{4\ve}{T}.\ee
	For $x\in G_K^\ve$ and $i\in\allN$ define $\x_i^x\in C([0,T],\R^J)$, for $t\in[0,T]$, by
	\begin{align}
		\label{eq:xx}\x_i^x(t):&= x+b(\alpha_0)t+\sigma(\alpha_0)w_i(t)\\
		\label{eq:xixcases}&=x+\bbr_1v_i\lb t\wedge\frac{T}{2}\rb+\bbr_2u\lb t-\frac{T}{2}\rb^+.
                	\end{align}
	Define $\y_i^x,\z_i^x\in\cts([0,T],\R^J)$, for $t\in[0,T]$, by 
	\begin{align}
		\label{eq:yx}\y_i^x(t)&:=\lb\ip{x,n_i}-\bbr_1\lb t\wedge\frac{T}{2}\rb\rb^- d_i(\alpha_0),\\
		\label{eq:zx}\z_i^x(t)&:=\x_i^x(t)+\y_i^x(t).
	\end{align}
Finally, also define 
	\be\label{eq:wtK}\wt K:=\lip_\sm(\alpha_0)\lb K+(K+\ve)\max_{j\in\allN}|v_j|+2\ve|u|\rb+\ve,\ee
where $\lip_\sm(\alpha_0)<\infty$ denotes the Lipschitz constant of the SM associated with the data $\{(d_j(\alpha_0),n_j),j\in\allN\}$; see Proposition \ref{prop:SP}.

\begin{lem}	\label{lem:11}
	The pair $(\z_i^x,\y_i^x)$ is the solution to the SP $\{(d_j(\alpha_0),n_j),j\in\allN\}$ for $\x_i^x$ on $[0,T]$ and the following inequalities hold:
	\begin{align}
	\label{eq:zixTwtKdelta}\sup_{s\in[0,T]}|\z_i^x(s)|+\sup_{s\in[0,T]}|\y_i^x(s)|&\le\wt K-\ve,
	\end{align} 
	and
	\begin{align}
	\label{eq:ipzixnj}\ip{\z_i^x(t),n_j}&\ge\ve,&& j\neq i,\quad t\in[0,T],\\
	\label{eq:ipziTnj}\ip{\z_i^x(T),n_j}&\ge2\ve,&& j\in\allN,\\
	\label{eq:ellixiT}|\y_i^x(T)|&\ge\ve.
	\end{align}
\end{lem}

We  first use this result to} prove Lemma \ref{lem:Pcup0T1}, and then present the proof of Lemma \ref{lem:11}.

\begin{proof}[Proof of Lemma \ref{lem:Pcup0T1}]
	For $i\in\allN$ define $A_i\subset\Omega$   by
	\be\label{eq:Ai}A_i:=\lcb\sup_{s\in[0,T]}|\bm(s)-w_i(s)|<\frac{\ve}{4(2\lip_\sm(\alpha_0)+1)(\norm{\sigma(\alpha_0)}+1)}\rcb.\ee
	Since Wiener measure assigns positive probability to nonempty open sets in $C([0,T],\R^J)$ (see, e.g., \cite{stroock1972support}) it follows that $\P(A_i)>0$. 
	By \eqref{eq:wi}, for $i\in\allN$,
		\be\label{eq:sups0Twi}\sup_{s\in[0,T]}|w_i(s)|\leq\wh K:=\norm{(\sigma(\alpha_0))^{-1}}\lb \frac{\bbr_1\max_{i\in\allN}|v_i|+\bbr_2|u|}{2}+|b(\alpha_0)|\rb T.\ee
Recall the definition of $\wt K$ from \eqref{eq:wtK}, let $c\in(0,1)$ satisfy
		\be\label{eq:c}0<c<\min\lb1,\frac{1}{2\lip_\sm(\alpha_0)}, \frac{\ve}{4(2\lip_\sm(\alpha_0)+1)(T+\wh K+\wt K)}\rb,\ee
and choose a neighborhood $U_{\alpha_0}$ of $\alpha_0$ in $U$ such that for all $\alpha\in U_0$,
	\begin{align}\label{eq:wtve}
		|b(\alpha)-b(\alpha_0)|+\norm{\sigma(\alpha)-\sigma(\alpha_0)}+\norm{(R(\alpha_0))^{-1}}\norm{R(\alpha)-R(\alpha_0)}<c.
	\end{align}
	We are left to show that for all $\alpha\in U_{\alpha_0}$, $x\in G_K^\ve$ and $i\in\allN$,
	\be\label{eq:Aisubset}A_i\subset\lcb\bigcup_{s\in[0,T]}\allN(\Z^{\alpha,x}(s))=\{i\},\Z^{\alpha,x}(T)\in G_{\wt K}^\ve\rcb.\ee

	We show  that on the set $A_i$, for any $\alpha\in U_{\alpha_0}$,
		\be\label{eq:ZYlip}\sup_{s\in[0,T]}|\Z^{\alpha,x}(s)-\z_i^x(s)|+\sup_{s\in[0,T]}|\Y^{\alpha,x}(s)-\y_i^x(s)|<\ve.\ee
	        It then follows from \eqref{eq:zixTwtKdelta} and \eqref{eq:ipzixnj}--\eqref{eq:ellixiT}
                of Lemma \ref{lem:11} that the RBM $\Z^{\alpha,x}$ does not hit the faces $F_j$, $j\ne i$, on the interval $[0,T]$, $\Z^{\alpha,x}(T)\in G_{\wt K}^\ve$, and $|\Y^{\alpha,x}(T)|>0$.
	Since $\Y^{\alpha,x}$ can only increase when $\Z^{\alpha,x}$ lies in $\partial G$, this will imply that, on the set $A_i$,
		\be\label{eq:allNZx}\bigcup_{s\in[0,T]}\allN(\Z^{\alpha,x}(s))=\{i\},\qquad\Z^{\alpha,x}(T)\in G_{\wt K}^\ve.\ee
 
	Define the $J$-dimensional $\{\F_t\}$-adapted processes $\wt\Y^{\alpha,x}$ and $\wt\X^{\alpha,x}$ by 
		\be\label{eq:wtYvey}\wt\Y^{\alpha,x}(t):= R(\alpha_0)\L^{\alpha,x}(t)=\Y^{\alpha,x}(t)+(R(\alpha_0)-R(\alpha))\L^{\alpha,x}(t),\qquad t\ge0.\ee
	and
	\be\label{eq:wtXvey}\wt\X^{\alpha,x}(t):= x+b(\alpha)t+\sigma(\alpha)\bm(t)-\wt\Y^{\alpha,x}(t)+\Y^{\alpha,x}(t),\qquad t\geq 0.\ee
        
	Then, due to the definition of $R(\alpha_0)$ in \eqref{eq:R} and the fact that the $i$th component of $\L^{\alpha,x}$ is nondecreasing and can only increase when $\Z^{\alpha,x}$ lies in face $F_i$ by Lemma \ref{lem:L}, it follows from the definition of the SP (Definition \ref{def:SP}) that $(\Z^{\alpha,x},\wt\Y^{\alpha,x})$ is the solution to the SP $\{(d_i(\alpha_0),n_i),i\in\allN\}$ for $\wt\X^{\alpha,x}$. 	
	Using the definition of $\wt\Y^{\alpha,x}$, the fact that $\Y^{\alpha,x}(\cdot)=R(\alpha)\L^{\alpha,x}(\cdot)$, the bound \eqref{eq:wtve},  the fact that $(g_i^x, h_i^x)$ solves the SP $\{(d_i(\alpha_0), n_i), i \in \allN\}$ for $f_i^x$ by Lemma \ref{lem:11}, the Lipschitz continuity of the SM, the bound on $g_i^x$ in \eqref{eq:zixTwtKdelta}, and our choice of $c>0$ in \eqref{eq:c},
          and the choice of neighborhood $U_{\alpha_0}$, which guarantees  \eqref{eq:wtve}, 
        we have, for $t\in[0,T]$,
	\begin{align}
	\label{eq:YYtilde}
		\sup_{s\in[0,t]}|\wt\Y^{\alpha,x}(s)-\Y^{\alpha,x}(s)|&\le\norm{R(\alpha_0)-R(\alpha)}\norm{(R(\alpha_0))^{-1}}\sup_{s\in[0,t]}|\wt\Y^{\alpha,x}(s)|\\
		\notag
		&\le c\lb\sup_{s\in[0,t]}|\wt\Y^{\alpha,x}(s)-\y_i^x(s)|+\sup_{s\in[0,t]}|\y_i^x(s)|\rb\\
		\notag
		&\le c\lb\lip_\sm(\alpha_0)\sup_{s\in[0,t]}|\wt\X^{\alpha,x}(s)-\x_i^x(s)|+\wt K\rb\\
		\notag
		&<\frac{1}{2}\sup_{s\in[0,t]}|\wt\X^{\alpha,x}(s)-\x_i^x(s)|+c\wt K.
	\end{align}
	By the definition of $\wt\X^{\alpha,x}$ in \eqref{eq:wtXvey}, the definition of $f_i^x$ in \eqref{eq:xx}, the bound \eqref{eq:wtve}, the bound on $w_i$ in \eqref{eq:sups0Twi} and the last display, for $t\in[0,T]$,
	\begin{align*}
		\sup_{s\in[0,t]}|\wt\X^{\alpha,x}(s)-f_i^x(s)|&\le\abs{b(\alpha)-b(\alpha_0)}t+(\norm{\sigma(\alpha_0)}+\norm{\sigma(\alpha)-\sigma(\alpha_0)})\sup_{s\in[0,t]}\abs{\bm(s)-w_i(s)}\\
		\notag
		&\qquad+\norm{\sigma(\alpha)-\sigma(\alpha_0)}\sup_{s\in[0,t]}\abs{w_i(s)}+\sup_{s\in[0,t]}|\wt\Y^{\alpha,x}(s)-\Y^{\alpha,x}(s)|\\
		\notag
		&< c(T+\wh K+\wt K)+(\norm{\sigma(\alpha_0)}+c)\sup_{s\in[0,T]}\abs{\bm(s)-w_i(s)}\\
		&\qquad+\frac{1}{2}\sup_{s\in[0,t]}|\wt\X^{\alpha,x}(s)-\x_i^x(s)|.
	\end{align*}
	Rearranging the last display and using the bound on $c$ in \eqref{eq:c}, we see that for $t\in[0,T]$,
	\begin{align*}
		\sup_{s\in[0,t]}|\wt\X^{\alpha,x}(s)-f_i^x(s)|&<\frac{\ve}{2(2\lip_\sm(\alpha_0)+1)}+2(\norm{\sigma(\alpha_0)}+1)\sup_{s\in[0,T]}\abs{\bm(s)-w_i(s)}.
	\end{align*}
	Therefore, on the set $A_i$ defined in \eqref{eq:Ai}, for $t\in[0,T]$, 
	\begin{align*}
		\sup_{s\in[0,t]}|\wt\X^{\alpha,x}(s)-f_i^x(s)|&<\frac{\ve}{2\lip_\sm(\alpha_0)+1}.
	\end{align*}
	Since $(\Z^{\alpha,x},\wt Y^{\alpha,x})$ and $(\z_i^x,\y_i^x)$ are the respective solutions to the SP $\{(d_j(\alpha_0),n_j),j\in\allN\}$ on $[0,T]$ for $\wt\X^{\alpha,x}$ and $\x_i^x$ (see Lemma \ref{lem:11}), using 
          the Lipschitz continuity of the SM (Proposition \ref{prop:SP}), the bound
        \eqref{eq:YYtilde} and the fact that \eqref{eq:c} implies $c\wt K < \ve/2$, 
          we have, for $t\in[0,T]$, 
	\begin{align*}
	&\sup_{s\in[0,t]}|\Z^{\alpha,x}(s)-\z_i^x(s)|+\sup_{s\in[0,t]}|\Y^{\alpha,x}(s)-\y_i^x(s)|\\
	&\qquad\leq\sup_{s\in[0,t]}|\Z^{\alpha,x}(s)-\z_i^x(s)|+\sup_{s\in[0,t]}|\wt\Y^{\alpha,x}(s)-\y_i^x(s)|+\sup_{s\in[0,t]}|\Y^{\alpha,x}(s)-\wt\Y^{\alpha,x}(s)|\\
	&\qquad\le\lb\lip_\sm(\alpha_0)+\frac{1}{2}\rb\sup_{s\in[0,t]}|\wt\X^{\alpha,x}(s)-\x_i^x(s)|+c\wt K \\
	&\qquad<\ve.
	\end{align*}
	We conclude that for $\alpha\in U_{\alpha_0}$, \eqref{eq:ZYlip} holds on the set $A_i$.
	Since $\P(A_i)>0$, this proves there exists $\wt\gamma_0>0$ such that \eqref{eq:wtgamma0} holds for all $\alpha\in U_{\alpha_0}$.  
	The extension to a uniform lower bound for all of $\alpha \in U_0$ then follows from the compactness of $U_0$.
\end{proof}

\begin{proof}[Proof of Lemma \ref{lem:11}]
	We first verify that $(\z_i^x,\y_i^x)$ is the solution to the SP $\{(d_j(\alpha_0),n_j),j\in\allN\}$ for $\x_i^x$ on $[0,T]$. 
	Since $\langle x, n_i \rangle \geq 0$ for $x \in G$, clearly $\z(0)=\x(0)$ and condition 1 of the SP holds.
	Next, by \eqref{eq:zixtni} and \eqref{eq:zixtnj}, we see that $\z_i^x(t)\in G$ for all $t\in[0,T]$, so condition 2 of the SP holds.
	Finally, by \eqref{eq:yx}, $\y_i^x$ is constant on $[T/2,T]$ and $\y_i^x(t)-\y_i^x(s)\ne0$ for $0\le s<t<T/2$ only if
	$$\ip{x,n_i}-\bbr_1t<0,$$
	in which case $\y_i^x(t)-\y_i^x(s)=\theta d_i$ for some $\theta \geq 0$
        and, by \eqref{eq:zx} and \eqref{eq:xx}, \eqref{eq:vi} and \eqref{eq:wtni}, 
	$$\ip{\z_x^i(t),n_i}=\ip{x,n_i}-\bbr_1t+\lb\ip{x,n_i}-\bbr_1t\rb^-=0,$$
	which implies $i\in\cup_{u\in(s,t]}\allN(\z_i^x(u))$.
	Thus, condition 3 of the SP holds.
	This concludes the verification.

	Next we prove \eqref{eq:zixTwtKdelta} holds.
	By \eqref{eq:xixcases}, \eqref{eq:r1r2}, the fact that $x\in G_K^\ve$, and \eqref{eq:wtK},
	\begin{align*}
	\sup_{t\in[0,T]}|\x_i^x(t)|&\leq|x|+(K+\ve)|v_i|+2\ve|u|\leq\frac{\wt K-\ve}{\lip_\sm(\alpha_0)}.
	\end{align*}
	Then by the Lipschitz continuity of the SM, the fact that $(0,0)$ is a solution to the SP $\{(d_j(\alpha_0),n_j),j\in\allN\}$ with input 0 (here $0$ denotes the $J$-dimensional function on $[0,\infty)$ that is identically zero), the definition of $\wt K$ in \eqref{eq:wtK}, we see that \eqref{eq:zixTwtKdelta} holds.

	Finally we prove \eqref{eq:ipzixnj}--\eqref{eq:ellixiT} hold.
	By \eqref{eq:xixcases}--\eqref{eq:zx}, \eqref{eq:wtni}--\eqref{eq:vi} and \eqref{eq:r1r2}, for $t\in[0,T]$,
	\begin{align}\label{eq:zixtni}
	\ip{\z_i^x(t),n_i}&=\ip{x,n_i}-\bbr_1\lb t\wedge\frac{T}{2}\rb+\bbr_2\lb t-\frac{T}{2}\rb^++\lb\ip{x,n_i}-\bbr_1\lb t\wedge\frac{T}{2}\rb\rb^-\\ \notag
	&\ge \bbr_2\lb t-\frac{T}{2}\rb^+,
	\end{align}
	and, for $j\neq i$ and $t\in[0,T]$,
	\begin{align}\label{eq:zixtnj}
	\ip{\z_i^x(t),n_j}&=\ip{x,n_j}+\bbr_1(\ip{d_i(\alpha_0),n_j})^-\lb t\wedge\frac{T}{2}\rb+\bbr_2\lb t-\frac{T}{2}\rb^+\\ \notag
	&\qquad+\lb\ip{x,n_i}-\bbr_1\lb t\wedge\frac{T}{2}\rb\rb^-\ip{d_i(\alpha_0),n_j}\\ \notag
	&\ge\ip{x,n_j}+\bbr_2\left(t-\frac{T}{2}\right)^-\\ \notag
	&\ge\ve+\frac{4\ve}{T}\lb t-\frac{T}{2}\rb^+,
	\end{align}
	where we have used $\bbr_2=4 \varepsilon/T$ and the fact that   $x\in G_K^\ve$ implies $\ip{x,n_j}\ge\ve$.
	It follows from the last two displays that \eqref{eq:ipzixnj} and \eqref{eq:ipziTnj} hold.
	Finally, by \eqref{eq:yx}, the fact that $x\in G_K^\ve$, \eqref{eq:r1r2} and the fact that $|d_i(\alpha)|\ge1$,
	\begin{align*}
	|\y_i^x(T)|&=\lb\ip{x,n_i}-K-\ve\rb^-|d_i(\alpha)| \ge\ve,
	\end{align*}
	so \eqref{eq:ellixiT} holds.
\end{proof}

\subsection{Proof of Proposition \ref{prop:taujfinite}}
\label{sec:lowerboundproof}

\begin{proof}[Proof of Proposition \ref{prop:taujfinite}]
	Fix $U_0$ and $K,T\in(0,\infty)$ as in the statement of the proposition.
	Let $S_i:=\frac{iT}{J+1}$ for $i\in\{0\}\cup\allN$. 
	By $J$ applications of Lemma \ref{lem:Pcup0T1}, there exist $K:= K_1\leq K_2\leq\cdots\leq K_J\leq K_{J+1}<\infty$ and $\gamma_1,\dots,\gamma_J>0$ such that for all $\alpha\in U_0$, $x\in G_{K_i}^\ve$ and $i\in\allN = \{1, \ldots, J\}$, 
		$$\P\lb\bigcup_{s\in[0,S_1]}\allN(\Z^{\alpha,x}(s))=\{i\},\;\Z^{\alpha,x}(S_1)\in G_{K_{i+1}}^\ve\rb\geq\gamma_i.$$
	Set $\gamma_0:=\gamma_1\times\cdots\times \gamma_J\times\gamma_{J+1}>0$, where
		$$\gamma_{J+1}:=\inf_{\alpha\in U_0}\inf_{z\in G_{K_{J+1}}^\ve}\P(\Z^{\alpha,z}(S_1)\in G_K)>0$$
	is positive because $(\alpha,z)\mapsto\Z^{\alpha,z}$ is continuous on the compact set $U_0\times G_{K_{J+1}}$ (see, e.g., \cite[Lemma 2.17]{Lipshutz2017}) and $\P(\Z^{\alpha,z}(S_1)\in G_K)>0$ for each $\alpha\in U_0$ and $z\in G_{K_{J+1}}$.
	By the strong Markov property for $\Z^{\alpha,x}$ and the last two displays, we have, for $\alpha\in U_0$ and $x\in G_K^\ve$,  
	\begin{align*}
	\P\lb\tau_K^{\alpha,x}(1)\le T\rb&\geq\P\lb\bigcup_{s\in[S_{i-1},S_i]}\allN(\Z^{\alpha,x}(s))=\{i\}\;\forall\;i\in\allN,\;\Z^{\alpha,x}(T)\in G_K\rb\\
	&\geq\prod_{i=1}^J\inf_{z\in G_{K_i}^\ve}\P\lb
	\Z^{\alpha,z}(S_1)\in G_{K_{i+1}}^\ve,\bigcup_{s\in[0,S_1]}\allN(\Z^{\alpha,z}(s))=\{i\}\rb\\
	&\qquad\times\inf_{z\in G_{K_{J+1}}^\ve}\P\lb\Z^{\alpha,z}(S_1)\in G_K\rb\\
	&\geq\gamma_0.
	\end{align*}
This proves \eqref{eq:tau1T} holds for all $\alpha\in U_0$ and $x\in G_K$.

	Let $\alpha\in U$ and $x\in G$.
	By the strong Markov property for $\Z^{\alpha,x}$, in order to prove that almost surely $\tau^{\alpha,x}(j)<\infty$ for all $j\in\N$, it suffices to show that almost surely $\tau^{\alpha,x}(1)<\infty$.
	Choose $K\in(0,\infty)$ such that $x\in G_K$ and a compact set $U_0$ in $U$ such that $\alpha\in U$.
	Let $T\in(0,\infty)$ be arbitrary.
	Define the sequence of stopping times $\{\xi^{\alpha,x}(j),j\in\N\}$ by $\xi^{\alpha,x}(0):=0$ and recursively set
		$$\xi^{\alpha,x}(j):=\inf\{t\ge\xi^{\alpha,x}(j-1)+T:\Z^{\alpha,x}(t)\in G_K\},\qquad j\in\N.$$
	Since $\Z^{\alpha,x}$ is positive recurrent (Theorem \ref{thm:RBMstable}), almost surely $\xi^{\alpha,x}(j)<\infty$ for each $j\in\N$.
	Thus, by the strong Markov property and \eqref{eq:tau1T},
	\begin{align*}
	\P\lb\tau_K^{\alpha,x}(1)=\infty\rb&\le\P\lb\xi^{\alpha,x}(j)<\infty,\;\tau_K^{\alpha,x}(1)\ge\xi^{\alpha,x}(j)+T,\;j\in\N\rb\\
	&\le\prod_{j=1}^\infty\P(\tau_K^{\alpha,x}(1)\ge\xi^{\alpha,x}(j)+T|\xi^{\alpha,x}(j)<\infty)\\
	&\le\lim_{j\to\infty}(1-\gamma)^j=0.
	\end{align*}
	This completes the proof.
\end{proof}

\section{Uniform stability of the joint processes}\label{sec:stable}

In this section we prove stability properties of the joint process.
The proof involves constructing a Lyapunov function for a discrete skeleton of the joint process.
Throughout this section we assume the data $\{(d_i(\cdot),n_i),i\in\allN\}$ satisfies Assumptions \ref{ass:independent}, \ref{ass:setB} and \ref{ass:projection}, and the coefficients $b(\cdot)$, $\sigma(\cdot)$ and $R(\cdot)$ satisfy Assumptions \ref{ass:holder} and \ref{ass:stable}. 
Also, as usual, $W$ is a $J$-dimensional $\{{\mathcal F}_t\}$-adapted Brownian motion on $(\Omega, {\mathcal F}, \mathbb{P})$.

\subsection{Lyapunov function for the RBM}

We first recall a Lyapunov function for the RBM that was introduced in \cite{Atar2001}.
For $\alpha\in U$ and $x\in G$, define $\x^{\alpha,x}\in\cts_G(\R^J)$ by 
	$$\x^{\alpha,x}(t):= x+b(\alpha)t,\qquad t\ge0,$$
and let $(\z^{\alpha,x},\y^{\alpha,x})$ denote the solution to the SP $\{(d_i(\alpha),n_i),i\in\allN\}$ for $\x^{\alpha,x}$.
Define
	\be\label{eq:RBMlyap}\lyap^\alpha(x):=\inf\{t\ge0:\z^{\alpha,x}(t)=0\}.\ee
In \cite{Atar2001} the function $\lyap^\alpha$ is used as a Lyapunov function to prove ergodicity of the RBM, and in \cite{Budhiraja2007} the exponential function $e^{r\lyap^\alpha}$, for some $r>0$, is used to prove geometric ergodicity of the RBM.
In the following lemma we recall some useful properties of $\lyap^\alpha$.
Let $\partial\C^\alpha$ denote the boundary of the cone $\C^\alpha$ defined in \eqref{eq:Calpha} and let $\rho(\cdot,\cdot)$ denote the Euclidean metric.
Recall from Assumption \ref{ass:stable} that $b(\alpha)\in(\C^\alpha)^\circ$ and so $\rho(b(\alpha),\partial\C^\alpha)>0$.
For $\alpha\in U$ define
	$$c(\alpha):=\frac{1}{\lip_\sm(\alpha)|b(\alpha)|}\qquad\text{and}\qquad C(\alpha):=\frac{4(\lip_\sm(\alpha))^3}{\rho(b(\alpha),\partial\C^\alpha)}.$$
Since $b(\cdot)$ and $d_i(\cdot)$, $i\in\allN$, are continuous on $U$, it follows from the definition of $\C^\alpha$ in \eqref{eq:Calpha} that $\alpha\mapsto\rho(b(\alpha),\partial\C^\alpha)$ is also continuous on $U$.
Combined with the fact that $\lip_\sm(\cdot)$ is also continuous (see Lemma \ref{lem:lipcont} and Remark \ref{rmk:lipcontinuous}), we see that both $c(\cdot)$ and $C(\cdot)$ are continuous on $U$.

\begin{lem} 
[{\cite[Lemmas 3.1 \& 4.1]{Atar2001}}]
\label{lem:TLyapunov}
For $\alpha\in U$ the following hold:
\begin{itemize}
	\item[(i)] For all $x,z\in G$,
		$$\abs{M^\alpha(x)-M^\alpha(z)}\le C(\alpha)\abs{x-z}.$$
	\item[(ii)] For all $x\in G$, $\abs{M^\alpha(x)}\ge c(\alpha)\abs{x}$.
	\item[(iii)] For all $x\in G$, $t\in[0,\infty)$ and $\Delta\in(0,\infty)$, a.s.
		$$\lyap^\alpha(\Z^{\alpha,x}(t+\Delta))\le(\lyap^\alpha(\Z^{\alpha,x}(t))-\Delta)^++C(\alpha)\lip_\sm(\alpha)\sup_{t\le s\le t+\Delta}\abs{\bm(s)-\bm(t)}.$$
\end{itemize}
\end{lem}

Given a compact subset $U_0$ in $U$, define $M_0:G\mapsto\R_+$ by
	\be\label{eq:lyap0}\lyap_0(x):=\sup_{\alpha\in U_0}M^\alpha(x),\qquad x\in G.\ee
Lemma \ref{lem:TLyapunov} ensures that $M_0$ is well defined, Lipschitz continuous and has compact level sets.
The following corollary is a consequence of Proposition \ref{prop:taujfinite}, {\color{blue} Remark \ref{rmk:tau}} and the fact that $M_0$ has compact level sets.
Recall the definition of $\tau^{\alpha,x}(1)=\tau_\infty^{\alpha,x}(1)$, the first time for the RBM to have visited all $(J-1)$-dimensional faces, as stated in \eqref{eq:tauj}.

\begin{cor}\label{cor:Ptau}
Given a compact subset $U_0$ in $U$ and $\Delta\in(0,\infty)$ there exists $\wh\gamma_0\in(0,1)$ such that $\P(\tau^{\alpha,x}(1)\le \Delta)\ge\wh\gamma_0$ for all $\alpha\in U_0$ and $x\in G$ satisfying $\lyap_0(x)\le \Delta$, where $\lyap_0$ is defined as in \eqref{eq:lyap0}.
\end{cor}

\subsection{Lyapunov function for the discrete skeleton of the joint process}

Given $\alpha \in U, x \in G, y \in G_x$, set $\xi := (x,y)$ and recall the joint process $(Z^{\alpha, x}, {\mathcal J}^{\alpha, \xi})$ defined in Theorem \ref{thm:derivativeprocess} and Section \ref{sec:joint}.
We define a Lyapunov function for the discrete skeleton of the joint process, which is composed of the Lyapunov function for the RBM that was introduced in \cite{Budhiraja1999} and the set $B^\alpha$ norm of the derivative process.
Without loss of generality, we can choose the sets $B^\alpha$, $\alpha\in U$, such that $\norm{y}_{B^\alpha}\le\abs{y}$ for all $\alpha\in U$ and $y\in\R^J$.
Given a compact subset $U_0$ in $U$, $M_0$ defined as in \eqref{eq:lyap0}, and $\eta,\Delta\in(0,\infty)$, define
\be\label{eq:stateK}\state_{\Delta,\eta}:=\lcb\xi=(x,y)\in\state:\max(\lyap_0(x),\eta^{-1}\abs{y})\le\Delta\rcb,\ee
where recall the definition of the parameter space $\state$ given in \eqref{eq:state}.
Since $\lyap_0$ has compact level sets by Lemma \ref{lem:TLyapunov}, $\state_{\Delta,\eta}$ is a relatively compact subset of $G\times\R^J$.

\begin{prop}\label{prop:lyapunov}
	Let $U_0$ be a compact subset in $U$ and define $\lyap_0$ as in \eqref{eq:lyap0}.
	There are constants $r_1,r_2\in(0,\infty)$, $\beta_0\in(0,1)$ and $\eta_0,\Delta_0,K_0\in(0,\infty)$ such that for each $\alpha\in U_0$, the Lyapunov function $V^\alpha:\state\mapsto[0,\infty)$ defined by
		\be\label{eq:Vdef}V^\alpha(x,y):=\exp\lb r_1\lyap^\alpha(x)+r_2\norm{y}_{B^\alpha}\rb,\qquad \xi=(x,y)\in\state,\ee
	satisfies
	\be\label{eq:drift}
		\E\lsb V^\alpha(\Xi^{\alpha,\xi}(\Delta_0))\rsb\le\beta_0 V^\alpha(\xi)+K_01\{\xi\in \state_{\Delta_0,\eta_0}\},\qquad\alpha\in U_0,\qquad\xi\in\state.
	\ee
\end{prop}

We have the following corollary.

\begin{cor}\label{cor:uniform}
For every compact set $U_0\subset U$ and relatively compact set $A\subset\state$,
	\be\sup\lcb\E\lsb V^\alpha(\Xib^{\alpha,\xi}(t))\rsb:t\ge0,\alpha\in U_0,\xi\in A\rcb<\infty.\ee
\end{cor}

The remaining subsections are devoted to the proofs of Proposition \ref{prop:lyapunov} and Corollary \ref{cor:uniform}.
Let $U_0$ be a compact subset that will remain fixed throughout the remaining subsections.

\subsection{Preparatory lemmas}

In order to define the constants $r_1,r_2\in(0,\infty)$ 
that appear in \eqref{eq:Vdef}, we need the following lemmas.
Let $\delta_0\in[0,1)$ be the contraction coefficient from Lemma \ref{lem:delta} and recall the definition of the stopping time $\tau^{\alpha,x}(1)=\tau_\infty^{\alpha,x}(1)$ given in \eqref{eq:tauj}.

\begin{lem}\label{lem:EphibTBbound}
	There exists $C_{\phib}\in(0,\infty)$ such that for $\alpha\in U_0$ and $\xi=(x,y)\in\state$, {and $T < \infty$,} 
	\be\label{eq:phibTwhtaubound}\norm{\phib^{\alpha,\xi}(T)}_{B^\alpha}\le\delta_0^{1{\{\tau^{\alpha,x}(1)\le T\}}}\norm{y}_{B^\alpha}+C_{\phib}\lb T+\norm{\bm}_{T}\rb.\ee
\end{lem}

The proof of Lemma \ref{lem:EphibTBbound} relies on the following lemma.

\begin{lem}
	[{\cite[Lemma 6.7]{Lipshutz2017}}]
	\label{lem:Lbound}
	There exists $C_L\in(0,\infty)$ such that for all $\alpha\in U$ and $x\in G$, almost surely $\norm{\L^{\alpha,x}}_T\le C_L(T+\norm{\bm}_T)$ for all $T\ge0$.
\end{lem}

\begin{proof}[Proof of Lemma \ref{lem:EphibTBbound}]
  By Remark \ref{rmk:derivativeprocessbeta}, $(\phib^{\alpha,\xi},\etab^{\alpha,\xi})$ is the solution to the derivative problem along $\Z^{\alpha,x}$ for $\psib^{\alpha,\xi}$ which, as defined in \eqref{eq:psib},
  satisfies  $\psib^{\alpha,\xi}(\cdot) =    \phib^{\alpha, x} (0) + b'(\alpha)\iota(\cdot) + \sigma' (\alpha) W(\cdot) + R'(\alpha) L^\alpha(\cdot)$, with $\phib^{\alpha, x} (0) = y$.   
	Thus, by Proposition \ref{prop:dpcontract}, the definition of $\tau^{\alpha,x}$ in \eqref{eq:tauj}, the definition of $\psib^{\alpha,\xi}$, and Lemma \ref{lem:Lbound}, we see that 
	\begin{align*}
	\norm{\phib^{\alpha,\xi}(T)}_{B^\alpha}&\le\delta_0^{1{\{\tau^{\alpha,x}(1)\le T\}}}\norm{\phib^{\alpha,\xi}(0)}_{B^\alpha}+\lip_\dm(\alpha)\sup_{0\le s\le T}\norm{\psib^{\alpha,\xi}(s)-\psib^{\alpha,\xi}(0)}_B\\
	&\le\delta_0^{1{\{\tau^{\alpha,x}(1)\le T\}}}\norm{y}_{B^\alpha}+\lip_\dm(\alpha)\lsb\norm{b'}_{B^\alpha}T+\norm{\sigma'}\norm{\bm}_{T}+\norm{R'}C_L(T+\norm{\bm}_{T})\rsb, 
	\end{align*}
	where 
		$$\norm{\sigma'}:=\sup\lcb\norm{\sigma' v}_{B^\alpha}:v\in\sphere,\alpha\in U_0\rcb\quad\text{and}\quad\norm{R'}:=\sup\lcb\norm{R' v}_{B^\alpha}:v\in\sphere,\alpha\in U_0\rcb.$$
	Thus, \eqref{eq:phibTwhtaubound} holds with $C_{\phib}:= \left[\sup_{\alpha\in U_0}\lip_\dm(\alpha)\right] \left(\max\{\norm{b'}_{B^\alpha},\norm{\sigma'}\}+\norm{R'}C_L\right)$, which is finite in view of Remark \ref{rmk:lipcontinuous} and the compactness of $U_0$.
\end{proof}

We need the following useful lemma.

\begin{lem}
	\label{lem:expWbound}
	For all $\lambda\in\R$ and $\Delta\in(0,\infty)$,
	$$\E\lsb\exp\lcb\lambda\norm{\bm}_\Delta\rcb\rsb\le\lb 1+\lambda\sqrt{\frac{8\Delta}{\pi}}\rb^J e^{2J\lambda^2\Delta}.$$
\end{lem}

\begin{proof}
  For each $j=1,\dots,J$, using the fact that $\sup_{0\le s\le 1}\bm^j(s)\buildrel{d}\over{=}|\bm(1)|$, and letting $\Phi$ denote the cumulative distribution function of the standard normal distribution, we have
\begin{align*}
\E\lsb\exp\lcb\lambda\sup_{0\le s\le 1}\bm^j(s)\rcb\rsb&=\E\lsb\exp\lcb\lambda\abs{\bm^j(1)}\rcb\rsb\\
&=2e^{\frac{\lambda^2}{2}}\lsb 1-\Phi\lb-\lambda\rb\rsb\\
&\le\lb 1+\sqrt{\frac{2}{\pi}}\lambda\rb e^{\frac{\lambda^2}{2}}
\end{align*}
where the last line  uses the well known inequality $2( 1- \Phi(-\lambda)) \leq (1 + \frac{2}{\pi} \lambda)$. 
Due to the fact that $|v|\le|v^1|+\cdots+|v^J|$ for all $v\in\R^J$, the Cauchy-Schwarz inequality, the facts that $\{\bm^1,\dots,\bm^J\}$ are independent and $\bm^j$ and $-\bm^j$ are equal in distribution for each $j$, we have
\begin{align*}
  \E\lsb\exp\lcb\lambda\sup_{0\le s\le 1}\abs{\bm(s)}\rcb\rsb &= \E\lsb \exp \lcb \lambda \sup_{0 \leq s \le 1} \abs{\bm^j(s)} \rcb \rsb^J \\
  &\le\E\lsb \exp\lcb\lambda\sup_{{0\le s\le 1}}\bm^j(s)+\lambda\sup_{{ 0\le s\le 1}}(-\bm^j(s))\rcb\rsb^J\\
&\le\E\lsb \exp\lcb2\lambda\sup_{{ 0\le s\le 1}}\bm^j(s)\rcb\rsb^J\\
&\le\lb 1+\sqrt{\frac{8}{\pi}}\lambda\rb^J e^{2J\lambda^2}.
\end{align*}
The lemma then follows from scaling properties of Brownian motion.
\end{proof}

\subsection{Proof of Proposition \ref{prop:lyapunov}}

\begin{proof}[Proof of Proposition \ref{prop:lyapunov}]
First, note that $V^\alpha$ is continuous and has compact level sets due to \eqref{eq:Vdef} and Lemma \ref{lem:TLyapunov}.
Next, define
\begin{equation}
  \label{def:C1}
        C_1 :=\sup\{C(\alpha)\lip_\sm(\alpha):\alpha\in U_0\}+\frac{1}{4}<\infty,\end{equation}
which is finite by the continuity of $C(\cdot)$ and $\lip_\sm(\cdot)$ on the compact set $U_0$ (see the discussion prior to Lemma \ref{lem:TLyapunov}). 
Let
\begin{equation}
  \label{def-delta0}
  \Delta_0:=\frac{2^7J^2C_1^2}{\pi}.
 \end{equation}
Let $\wh\gamma_0\in(0,1)$ be as in Corollary \ref{cor:Ptau}, and let $\delta_0\in(0,1)$ be the contraction
coefficient from Lemma \ref{lem:delta} (and hence, Lemma \ref{lem:EphibTBbound}).
Choose
	\be\label{eq:c1c2}r_1\in\lb0,\frac{1}{4JC_1^2}\rb\ee
sufficiently small so that
	\be\label{eq:r1}\exp\lcb\frac{5r_1\Delta_0}{4} \rcb\lb 1+r_1C_1\sqrt{\frac{32\Delta_0}{\pi}}\rb^{\frac{J}{2}}\le\lb1-\wh\gamma_0\rb^{-\frac{1}{8}}.\ee
        Taking a Taylor expansion of the exponential function and substituting in the definition of $\Delta_0$ from \eqref{def-delta0}  shows that 
	\be\label{eq:Tdef}1+r_1C_1\sqrt{\frac{8\Delta_0}{\pi}}\le\exp\lcb r_1C_1\sqrt{\frac{8\Delta_0}{\pi}}\rcb=\exp\lcb\frac{r_1\Delta_0}{4J}\rcb.\ee
Set
	$$r_2:=\frac{r_1}{4C_{\phib}}.$$
Choose $\eta_0\in(0,\infty)$ sufficiently large so that
	\be\label{eq:eta0}\lb1-\wh\gamma_0\lsb1-\exp\lb-2r_2\eta_0(1-\delta_0)\rb\rsb\rb^{\frac{1}{4}}\le(1-\wh\gamma_0)^{\frac{1}{8}}.\ee
Let
	\be\label{eq:beta}\beta_0:=\max\lcb \exp\lb-\frac{r_1\Delta_0}{4}\rb,\lb1-\wh\gamma_0\rb^{\frac{1}{8}}\rcb\in(0,1).\ee

Given $\xi = (x,y)\in\state$, by the definition of $V^\alpha$ in \eqref{eq:Vdef}, the bound on $\lyap^\alpha(\Z^x(\Delta_0))$ given in part (iii) of Lemma \ref{lem:TLyapunov}, and the bound on $\norm{\phib^\xi(\Delta_0)}_{B^\alpha}$ given in Lemma \ref{lem:EphibTBbound},  \eqref{def:C1} and the definition of $r_2$,  we have 
\begin{align}
	\label{eq:EVT}
	V^\alpha(\Xib^\xi(\Delta_0))&=\exp\lb r_1\lyap^\alpha(\Z^{\alpha,x}(\Delta_0))+r_2\norm{\phib^{\alpha,\xi}(\Delta_0)}_{B^\alpha}\rb\\
	\notag
	&\le\exp\lcb r_1(\lyap^\alpha(x)-\Delta_0)^++r_2\delta_0^{1{\{\tau^{\alpha,x}(1)\le\Delta_0\}}}\norm{y}_{B^\alpha}+\frac{r_1}{4}\Delta_0+r_1C_1\norm{\bm}_{\Delta_0})\rcb.
\end{align}
We treat the following three mutually exclusive and exhaustive cases separately.

\textit{Case 1}: $\lyap_0(x)>\Delta_0$. \\
The following inequalities are explained below:
\begin{align*}
	\E\lsb V^\alpha(\Xib^\xi(\Delta_0))\rsb&\le V^\alpha(\xi)\exp\lcb-\frac{3r_1\Delta_0}{4}\rcb\E\lsb\exp\lcb r_1C_1\norm{\bm}_{\Delta_0}\rcb\rsb\\
	&\le V^\alpha(\xi)\exp\lcb \lb-\frac{3}{4}+2Jr_1C_1^2\rb r_1\Delta_0\rcb\lb 1+r_1C_1\sqrt{\frac{8\Delta_0}{\pi}}\rb^J\\
	&\le V^\alpha(\xi)\exp\lb{-\frac{r_1\Delta_0}{2}}\rb\lb 1+r_1C_1\sqrt{\frac{8\Delta_0}{\pi}}\rb^J\\
	&\le V^\alpha(\xi)\exp\lb{-\frac{r_1\Delta_0}{4}}\rb\\
	&\le\beta_0 V^\alpha(\xi).
\end{align*}
The first inequality is due to \eqref{eq:EVT}, the definition of $V^\alpha(\xi)$ in \eqref{eq:Vdef} and the
inequality $\delta_0 < 1$.
The second inequality follows from Lemma \ref{lem:expWbound} with $\lambda=r_1C_1$ and $\Delta = \Delta_0$.
The third holds because of the  restriction on  $r_1$ in \eqref{eq:c1c2}.
The fourth inequality is due to \eqref{eq:Tdef}. 
The final inequality follows from the definition of $\beta_0$ in \eqref{eq:beta}. 

\textit{Case 2}: $\lyap_0(x)\le \Delta_0$ and $\norm{y}_{B^\alpha}>\eta_0\Delta_0$. \\
The following inequalities are explained below:
\begin{align*}
	&\E\lsb V^\alpha(\Xib^\xi(\Delta_0))\rsb\\
	&\qquad\le V^\alpha(\xi)\exp\lcb\frac{r_1\Delta_0}{4}\rcb\E\lsb\exp\lcb-r_2\eta_0 \Delta_0\lb1-\delta_0^{1\{\tau^{\alpha,x}(1)\le\Delta_0\}}\rb+r_1C_1\norm{\bm}_{\Delta_0}\rcb\rsb\\
	&\qquad\le V^\alpha(\xi)\exp\lcb\frac{r_1\Delta_0}{4}\rcb\E\lsb\exp\lcb 2r_1C_1\norm{\bm}_{\Delta_0}\rcb\rsb^{\frac{1}{2}}\E\lsb\exp\lcb-2r_2 \eta_0 \Delta_0\lb1-\delta_0^{1{\{\tau_1(x)\le \Delta_0\}}}\rb\rcb\rsb^{\frac{1}{2}}\\
	&\qquad\le V^\alpha(\xi)\exp\lcb \frac{r_1\Delta_0}{4}+4Jr_1^2C_1^2\Delta_0 \rcb\lb 1+2r_1C_1\sqrt{\frac{8\Delta_0}{\pi}}\rb^{\frac{J}{2}}\lsb1-\wh\gamma_0\lb1-e^{-2r_2 \eta_0 \Delta_0(1-\delta_0)}\rb\rsb^{\frac{1}{2}}\\
	&\qquad\le V^\alpha(\xi)\exp\lcb\frac{5r_1\Delta_0}{4} \rcb\lb 1+r_1C_1\sqrt{\frac{32\Delta_0}{\pi}}\rb^{\frac{J}{2}}\lb1-\wh\gamma_0\rb^{\frac{1}{4}}\\
	&\qquad\le \beta_0 V^\alpha(\xi).
\end{align*}
The first inequality is due to \eqref{eq:EVT}, the definition of $V^\alpha(\xi)$ in \eqref{eq:Vdef} and the fact that $\norm{y}_{B^\alpha}\ge {\eta_0} \Delta_0$.
The second inequality follows from the Cauchy-Schwarz inequality.
The third inequality is due to Lemma \ref{lem:expWbound} with $\lambda=2r_1C_1$ {and $\Delta = \Delta_0$}, and Corollary \ref{cor:Ptau}.
The fourth inequality is due to our choice of $\eta_0\in(0,\infty)$ so that \eqref{eq:eta0} holds, and the restriction on $r_1$ in \eqref{eq:c1c2}.  
The final inequality follows from the inequality \eqref{eq:r1} and the definition of $\beta_0$ in \eqref{eq:beta}.

\textit{Case 3}: $M_0(x) \leq \Delta_0$ and $\norm{y}_{B^\alpha} \leq \eta_0 \Delta_0$ \\
       {Using the fact that the norms $\norm{y}_{B^\alpha}$ and  $|y|$ are equivalent for every $\alpha$ and the  
         map  $\alpha \mapsto B^\alpha$ is continuous 
  in the Hausdorff metric, as assumed in Assumption \ref{ass:setB}, it follows that there exists  $\bar{\eta}(x) < \infty$ such that
  $\sup\{|y|: y \in \hyper_x, \norm{y}_{B^\alpha}  \leq \eta_0 \mbox{ for some } \alpha \in U_0 \} \leq \bar{\eta} (x)$, where the  homogeneity of the norm functional and the fact that the supremum is being taken over  a  hyperplane to
  justify the finiteness of $\bar{\eta}(x)$.
  Since there are only a finite number of hyperplanes $\hyper_x$ as $x$ ranges over $G$, we see that
  $\bar{\eta} := \max_{x \in G} \bar{\eta}(x)$ is also finite. 
  Thus, the condition of this  case implies that we always have 
  $\xi := (x,y) \in \state_{\Delta_0, \bar{\eta}}$,  where 
$\state_{\Delta_0, \bar{\eta}}$ is as  defined in \eqref{eq:stateK}. }  
Then, using \eqref{eq:EVT} and the definition of $V^\alpha(\xi)$ in \eqref{eq:Vdef}, applying Lemma \ref{lem:expWbound} with $\lambda=r_1C_1$ and $\Delta = \Delta_0$,  noting the definition of $r_1$ in \eqref{eq:c1c2} and the fact that  $\delta_0 \leq  1$, we obtain 
\begin{align*}
	\E\lsb V^\alpha(\Xib^\xi(\Delta_0))\rsb&\le V^\alpha(\xi){\E}\lsb\exp\lcb \frac{r_1\Delta_0}{4}+r_1 C_1\norm{\bm}_{\Delta_0}\rcb\rsb\le K_0,
\end{align*} 
where 
$$K_0:=\exp\lcb\frac{r_1\Delta_0}{4}+2Jr_1^2C_1^2\Delta_0\rcb\lb 1+r_1C_1\sqrt{\frac{8\Delta_0}{\pi}}\rb^J\sup_{\alpha\in U_0}\sup_{\wt\xi\in\state_{\Delta_0,\bar{\eta}}}V^\alpha(\wt\xi)$$
{is finite due to the continuity of $V^\alpha$ and the compactness of $U_0$ and relative compactness of} $\state_{\Delta_0, \bar{\eta}_0}$.    
	
Combining the three cases yields \eqref{eq:drift}.  
\end{proof}

\subsection{Proof of Corollary \ref{cor:uniform}}
\label{sec:stableproof}

We first prove a useful lemma.

\begin{lem}
\label{lem:Vbound0Delta}
Let $\Delta\in(0,\infty)$.
For any $\alpha\in U$ and $\xi\in\state$,
$$\sup_{t\in[0,\Delta]}\E\lsb V^\alpha(\Xib^{\alpha,\xi}(t))\rsb\le V^\alpha(\xi)\lb1+C_V\sqrt{\frac{8\Delta}{\pi}}\rb^J\exp\lcb\lb 2JC_V^2+r_2C_{\phib}\rb \Delta\rcb,$$  
where $C_V:=r_1C(\alpha)\lip_\sm(\alpha)+r_2C_{\phib}$.
\end{lem}

\begin{proof}
Recalling the definition of $V^\alpha$ in \eqref{eq:Vdef}, and applying Lemma \ref{lem:TLyapunov}(iii) with $t=0$ and $\Delta = t$, Lemma \ref{lem:EphibTBbound} with $T = t$, and Lemma \ref{lem:expWbound} with $\lambda = C_V$ and  $\Delta =t$, we see that for $t\in[0,\Delta]$, 
\begin{align*}
	\E\lsb V^\alpha(\Xib^{\alpha,\xi}(t))\rsb&=\E\lsb\exp\lcb r_1\lyap_0(\Z^{\alpha,x}(t))+r_2\norm{\phib^{\alpha,\xi}(t)}_{B^\alpha}\rcb\rsb\\
	&\le V^\alpha(\xi)\E\lsb\exp\lcb C_V\norm{\bm}_t+r_2C_{\phib}t\rcb\rsb\\
	&\le V^\alpha(\xi)\lb1+C_V\sqrt{\frac{8t}{\pi}}\rb^J\exp\lcb (2J C_V^2+r_2C_{\phib})t\rcb.
\end{align*} 
Taking the supremum  over $t\in[0,\Delta]$ on both sides completes the proof of the lemma.
\end{proof}

\begin{proof}[Proof of Corollary \ref{cor:uniform}]
	Fix a compact set $U_0$ in $U$ and a relatively compact set $A\subset\state$, and let $r_1,r_2\in(0,\infty)$, $\beta_0\in(0,1)$ and $\Delta_0,\eta_0,K_0\in(0,\infty)$ be as in Proposition \ref{prop:lyapunov}.
	Let $\alpha\in U_0$ and $\xi\in A$.  
By the Markov property for $\Xib^{\alpha,\xi}$ ({see Theorem \ref{thm:jointmarkov}}) and Proposition \ref{prop:lyapunov}, we have, for each $n\in\N$,
\begin{align*}
	\E\lsb V^\alpha(\Xib^{\alpha,\xi}(n\Delta_0))\rsb\le\beta_0\E\lsb V^\alpha(\Xib^{\alpha,\xi}((n-1)\Delta_0))\rsb+K_0.
\end{align*}
Recursively applying the last display yields, for $n\in\N$,
\begin{align*}
	\E\lsb V^\alpha(\Xib^{\alpha,\xi}(n\Delta_0))\rsb\le\beta_0^n V^\alpha(\xi)+\sum_{k=1}^{n-1}\beta_0^kK_0\le V^\alpha(\xi)+\frac{K_0}{1-\beta_0}.
\end{align*}
Another application of the Markov property for $\Xib^{\alpha,\xi}$, when combined with  Lemma \ref{lem:Vbound0Delta}, shows that for all $t\in[n\Delta_0,(n+1)\Delta_0]$,
\begin{align*}
	\E\lsb V^\alpha(\Xib^{\alpha,\xi}(t))\rsb&\le \E\lsb V^\alpha(\Xib^{\alpha,\xi}(n\Delta_0))\rsb\lb1+C_V\sqrt{\frac{8\Delta_0}{\pi}}\rb^J\exp\lcb\lb 2JC_V^2+r_2C_{\phib}\rb\Delta_0\rcb\\
	&\le\lb V^\alpha(\xi)+\frac{K_0}{1-\beta_0}\rb\lb1+C_V\sqrt{\frac{8\Delta_0}{\pi}}\rb^J\exp\lcb\lb 2J C_V^2+r_2C_{\phib}\rb \Delta_0\rcb.
\end{align*}
To complete the proof, we 
take the supremum over $t\in[n\Delta_0,(n+1)\Delta_0]$, $n\in\N_0$, $\alpha\in U_0$ and $\xi\in A$ on both sides of the previous inequality, and invoke the continuity of the map $(\alpha, \xi) \mapsto V^{\alpha}(\xi)$ and the compactness of
$U_0$ and $A$ to conclude the  finiteness of the right-hand side. 
\end{proof}

\section{Ergodicity of the joint process}
\label{sec:ergodic}

Throughout this section we assume the data $\{(d_i(\cdot),n_i),i\in\allN\}$ satisfies Assumptions \ref{ass:independent}, \ref{ass:setB} and \ref{ass:projection}, and the coefficients $b(\cdot)$, $\sigma(\cdot)$ and $R(\cdot)$ satisfy Assumptions \ref{ass:holder} and \ref{ass:stable}. 

\subsection{Uniqueness of the stationary distribution}
\label{sec:jointunique}

In this section we prove there is at most one stationary distribution for the joint process. 
In the next section we prove existence of a stationary distribution.
The proof of uniqueness is nonstandard due to degeneracy of the joint process, which is a $2J$-dimensional process driven by a $J$-dimensional Brownian motion. 
We use the asymptotic coupling method formalized by Hairer, Mattingly and Scheutzow in \cite{Hairer2011}.

In order to describe this method, we first need some notation.
Let ${\bf X}$ be a separable metric space with metric $\text{dist}(\cdot,\cdot)$.
Let $\mathcal{M}(\dr({\bf X}))$ and $\mathcal{M}(\dr({\bf X})\times\dr({\bf X}))$ denote the set of probability measures on $(\dr({\bf X}),\B(\dr({\bf X})))$ and $(\dr({\bf X})\times\dr({\bf X}),\B(\dr({\bf X}))\otimes\B(\dr({\bf X})))$, respectively.
For $m_1,m_2\in\mathcal{M}(\dr({\bf X}))$ let $\mathcal{C}(m_1,m_2)$ denote the set of \emph{couplings} of $m_1$ and $m_2$; that is,
	$$\mathcal{C}(m_1,m_2):=\lcb\Upsilon\in\mathcal{M}(\dr({\bf X})\times\dr({\bf X})):\Upsilon(\cdot\times\dr({\bf X}))=m_1(\cdot),\;\Upsilon(\dr({\bf X})\times\cdot)=m_2(\cdot)\rcb.$$
Define the \emph{diagonal at infinity} by
	$$\mathcal{D}:=\lcb (\zeta_1,\zeta_2)\in\dr({\bf X})\times\dr({\bf X}):\lim_{t\to\infty}\text{dist}(\zeta_1(t),\zeta_2(t))=0\rcb.$$
We say $\Upsilon\in\mathcal{C}(m_1,m_2)$ is an \emph{asymptotic coupling} of $m_1$ and $m_2$ if $\Upsilon(\mathcal{D})=1$.
The following theorem is a continuous version of \cite[Theorem 1.1]{Hairer2011} (see, for example, \cite[Proposition 5.1]{AghRam19a} and \cite[Appendix C]{AghRam19b}), where we have also relaxed the requirement that ${\bf X}$ be \emph{complete}.
A careful examination of the proof of \cite[Theorem 1.1]{Hairer2011} reveals that the result still holds even if the metric space is not complete.
Also note that \cite[Theorem 1.1]{Hairer2011} proves a stronger result related to \emph{equivalent} asymptotic couplings, which we do not need here.
The version of the method we use is closely related to the \emph{asymptotic flatness} condition for the stochastic flow of a diffusion introduced by Basak and Bhattacharya \cite{Basak1992} to prove uniqueness of the stationary distribution for a degenerate diffusion.
Indeed, in the terminology of \cite{Basak1992}, we use a coupling construction to 
show that the stochastic flow of the (degenerate) joint process is almost surely asymptotically flat;  
see \eqref{eq:asymcoupled} in the proof of Theorem \ref{thm:unique} below. 

\begin{theorem}\label{thm:coupling}
Let $\{\mathcal{P}_t\}=\{\mathcal{P}_t,t\ge0\}$ be a Markov transition semigroup on a separable metric space ${\bf X}$ admitting two stationary distributions $\mu_1$ and $\mu_2$.
For $i=1,2$ let $P_{\mu_i}$ denote the distribution of the Markov process with initial distribution $\mu_i$ and transition semigroup $\{\mathcal{P}_t\}$ on $(\dr({\bf X}),\B(\dr({\bf X})))$.
Suppose there is an asymptotic coupling of $P_{\mu_1}$ and $P_{\mu_2}$.
Then $\mu_1=\mu_2$.
\end{theorem}

With Theorem \ref{thm:coupling} in hand, we state and prove that there is at most one stationary distribution for the joint process.  

\begin{theorem}
\label{thm:unique}
	For each $\alpha\in U$ there is at most one stationary distribution for the joint process $\Xib^\alpha$.
\end{theorem}

\begin{proof}
Throughout this proof we fix $\alpha\in U$ and suppress the $\alpha$ dependence.
Suppose there are two stationary distributions $\mu_1$ and $\mu_2$ for the joint process. 
For $i=1,2,$ let $P_{\mu_i}$ denote the distribution of the joint process with initial distribution $\mu_i$.
We construct the asymptotic coupling of $P_{\mu_1}$ and $P_{\mu_2}$ as follows.
Due to the uniqueness of the stationary distribution of the RBM stated in Theorem \ref{thm:RBMstable}, the first marginals of $\mu_1$ and $\mu_2$ must be equal in the sense that
	\be\label{eq:c1c2umarignal}\mu_1((A\times\R^J)\cap\state)=\mu_2((A\times\R^J)\cap\state),\qquad A\in\B(G).\ee
Let $\Xib_1=(\Z_1,\phib_1)$ and $\Xib_2=(\Z_2,\phib_2)$ denote the joint processes with respective initial distributions $\mu_1$ and $\mu_2$, and common driving Brownian motion $\bm$ such that $\Z_1(0)$ and $\Z_2(0)$ are independent of $\bm$.
Then $P_{\mu_i}(\cdot)=\P(\Xib_i\in\cdot)$ for $i=1,2$.
In view of \eqref{eq:c1c2umarignal}, we can assume that $\Xib_1$ and $\Xib_2$ are built on the common probability space $(\Omega,\F,\P)$ such that a.s.\ $\Z_1(0)=\Z_2(0)$.
Let $\Y_1$ and $\Y_2$ denote the respective constraining processes and set $\L_1(\cdot)=R^{-1}\Y_1(\cdot)$ and $\L_2(\cdot)=R^{-1}\Y_2(\cdot)$ as in \eqref{eq:Lalpha}.
By the pathwise uniqueness of RBMs (Theorem \ref{thm:rbm}) and Remark \ref{rmk:pathwiseunique}, a.s.\ $\Z_1=\Z_2$ and $\L_1=\L_2$.
Define the coupling $\Upsilon$ on $(\dr(\state)\times\dr(\state),\B(\dr(\state))\otimes\B(\dr(\state)))$ of the probability measures $P_{\mu_1}$ and $P_{\mu_2}$ on $(\dr(\state),\B(\dr(\state)))$ by
	$$\Upsilon\left(\mathcal{A}_1\times \mathcal{A}_2\right)=\P\left((\Xib_1,\Xib_2)\in \mathcal{A}_1\times \mathcal{A}_2\right),\qquad \mathcal{A}_1, \mathcal{A}_2\in\B(\dr(\state)).$$
For $i=1,2$ define $\psib_i$ as in \eqref{eq:psib}, but with $\psib_i$ and $\phib_i$ in place of $\psib$ and $\phib$, respectively.
Since $\psib_1$ and $\psib_2$ are driven by the same  Brownian motion $W$ and a.s.\ $\L_1=\L_2$, it follows from \eqref{eq:psib} that 
\be\label{eq:psibdiffphiinfdiff}
	\psib_1(t)-\psib_2(t)=\phib_1(0)-\phib_2(0),\qquad t\ge0.
\ee
By Remark \ref{rmk:derivativeprocessbeta}, Definition \ref{def:dp}, the fact that a.s.\ $\Z_1=\Z_2$, and the linearity of the derivative map (Lemma \ref{lem:dmlinear}), a.s.
\be\label{eq:phibdiffdmZpsibdiff}
	\phib_1-\phib_2=\dm_{\Z_1}(\psib_1)-\dm_{\Z_1}(\psib_2)=\dm_{\Z_1}(\psib_1-\psib_2).
\ee
Due to \eqref{eq:phibdiffdmZpsibdiff}, \eqref{eq:psibdiffphiinfdiff}, the definition of $\tau_j(x)$ in \eqref{eq:tauj},   
a repeated application of the bound in Proposition \ref{prop:dpcontract} with $S = \min(t,\tau_j(Z_1(0)))$ and $T = \min(t,\tau_{j+1}(Z_1(0)))$,   Proposition \ref{prop:taujfinite} and Remark \ref{rmk:tau}  show that  we have a.s. 
\begin{align}\label{eq:phicoupling} 
	\lim_{t\to\infty}\norm{\phib_1(t)-\phib_2(t)}_B&\leq\norm{\phib_1(0)-\phib_2(0)}_B\prod_{j=1}^\infty\delta_0^{1{\lcb\tau_j(\Z_1(0))<\infty\rcb}}=0, 
\end{align}
where $\delta_0 \in (0,1)$ is the contraction coefficient from  Lemma \ref{lem:delta}. 
Let
	$$\mathcal{D}:=\lcb(\zeta_1,\zeta_2)\in\dr(\state)\times\dr(\state):\lim_{t\to\infty}\norm{\zeta_1(t)-\zeta_2(t)}=0\rcb.$$
Since a.s.\ $\Z_1=\Z_2$, \eqref{eq:phicoupling} implies that
	\be\label{eq:asymcoupled}\Upsilon(\mathcal{D})=\P\left(\lim_{t\to\infty}\norm{\Xib_1(t)-\Xib_2(t)}=0\right)=\P\left(\lim_{t\to\infty}\norm{\phib_1(t)-\phib_2(t)}_{B^\alpha}=0\right)=1.\ee
Therefore, $\Upsilon$ is an asymptotic coupling of $P_{\mu_1}$ and $P_{\mu_2}$, and Theorem \ref{thm:coupling} implies $\mu_1=\mu_2$.
\end{proof}

\subsection{Proof of Theorem \ref{thm:stable}}
\label{sec:proofstable}

Given $\alpha\in U$, $\xi\in\state$ and $t>0$, define the probability measure $Q_t^{\alpha,\xi}$ on $\state$ by
	$$Q_t^{\alpha,\xi}(A):=\frac{1}{t}\int_0^tP_s^\alpha(\xi,A)ds,\qquad A\in\B(\state),$$
where $P_s^\alpha(\xi,A)$ is the transition function defined in \eqref{eq:Pt}.
With Corollary \ref{cor:uniform} in hand, the proof of existence of a stationary distribution follows a standard argument (see, e.g., the proof of \cite[Theorem 1.2]{Billingsley1999}), with the main difference being that the state space for the joint process $\state$ is not complete.

\begin{proof}[Proof of Theorem \ref{thm:stable}]
Fix $\xi_0\in\state$.
By Corollary \ref{cor:uniform},
	\be\label{eq:mdef}m:=\sup_{t\ge0}\E\lsb V^\alpha(\Xib^{\alpha,\xi_0}(t))\rsb<\infty.\ee
Thus, by Markov's inequality, for all $t\ge0$ and $K<\infty$,
	$$Q_t^{\alpha,\xi_0}\lb\lcb\xi\in\state:V^\alpha(\xi)\ge K\rcb\rb\le\frac{m}{K}.$$
Since $V^\alpha$ has compact level sets, it follows that the family of probability measures $\{Q_t^{\alpha,\xi_0}\}_{t\ge0}$ on the Polish space $G\times\R^J$ is tight.
Let $\mu$ denote any weak limit point.
By Theorem \ref{thm:RBMstable} and the fact that the renormalized occupation measures of the RBM converge to its unique stationary distribution (see, e.g., \cite[Chapter 4, Theorem 9.3]{EK2009}),
the first marginal of $\mu$ is the unique stationary distribution for the RBM, which, by \cite[Theorem 2]{Kang2014a}, is supported on $G^\circ$. 
Thus, $\mu$ is supported on $G^\circ\times\R^J\subset\state$. 
Let $s>0$ and $g:\state\mapsto\R$ be a bounded and continuous function.
Let $\ve>0$.
Since $Q_t^{\alpha,\xi_0}$ converges to $\mu$ in the weak topology and $(P_s^\alpha g):\state\mapsto\R$ is a bounded and continuous function by the Feller continuity shown in Theorem \ref{thm:jointmarkov}, we can choose $t\ge 2s\norm{g}_\infty/\ve$ sufficiently large so that
	$$\abs{\int_\state (P_s^\alpha g)(\xi)\mu(d\xi)-\int_\state (P_s^\alpha g)(\xi)Q_t^{\alpha,\xi_0}(d\xi)}+\abs{\int_\state g(\xi)Q_t^{\alpha,\xi_0}(d\xi)-\int_\state g(\xi)\mu(d\xi)}<\ve.$$
For such $t\ge 2s\norm{g}_\infty/\ve$, we have
\begin{align*}
	\abs{(\mu P_s^\alpha)(g)-\mu(g)}&\le\abs{\int_\state g(\xi)(\mu P_s^\alpha)(d\xi)-\int_\state g(\xi)(Q_t^{\alpha,\xi_0}P_s^\alpha)(d\xi)}\\
	&\qquad+\abs{\int_\state g(\xi)(Q_t^{\alpha,\xi_0}P_s^\alpha)(d\xi)-\int_\state g(\xi)Q_t^{\alpha,\xi_0}(d\xi)}\\
	&\qquad+\abs{\int_\state g(\xi)Q_t^{\alpha,\xi_0}(d\xi)-\int_\state g(\xi)\mu(d\xi)}\\
	&\le\abs{\int_\state (P_s^\alpha g)(\xi)\mu(d\xi)-\int_\state (P_s^\alpha g)(\xi)Q_t^{\alpha,\xi_0}(d\xi)}\\
	&\qquad+\frac{1}{t}\abs{\int_t^{t+s} (P_u^\alpha g)(\xi)du-\int_0^s (P_u^\alpha g)(\xi)du}\\
	&\qquad+\abs{\int_\state g(\xi)Q_t^{\alpha,\xi_0}(d\xi)-\int_\state g(\xi)\mu(d\xi)}\\
	&\le 2\ve.
\end{align*} 
Since $\ve>0$ was arbitrary, it follows that $\mu$ is a stationary distribution for the joint process, which is unique by Theorem \ref{thm:unique}.
\end{proof}

\section{Sensitivities of the stationary distribution of an RBM}\label{sec:interchange}

In this section we prove Theorem \ref{thm:sensitivity}. 
Throughout this section we assume the data $\{(d_i(\cdot),n_i),i\in\allN\}$ satisfies Assumptions \ref{ass:independent}, \ref{ass:setB} and \ref{ass:projection}, and the coefficients $b(\cdot)$, $\sigma(\cdot)$ and $R(\cdot)$ satisfy Assumptions \ref{ass:holder} and \ref{ass:stable}. 
Fix a continuous differentiable function $f:G\mapsto\R$ with bounded and continuous Jacobian $f':G\mapsto\R^{1\times J}$.
Let $x\in G$ and $\xi=(x,0)\in\state$.
For each $t>0$ define the function $\theta_t:U\mapsto\R$ by
	$$\theta_t(\alpha):=\frac{1}{t}\int_0^t\E\lsb f(\Z^{\alpha,x}(s))\rsb ds,\qquad\alpha\in U.$$
By Corollary \ref{cor:pathwise}, for each $t\ge0$, $\theta_t(\cdot)$ is differentiable on $U$ with
	$$\theta_t'(\alpha)=\frac{1}{t}\int_0^t\E\lsb f'(\Z^{\alpha,x}(s))\phib^{\alpha,\xi}(s)\rsb ds,\qquad\alpha\in U.$$
Then by Theorem \ref{thm:stable} and Corollary \ref{cor:uniform},
\begin{align}
\label{eq:stable}
	\lim_{t\to\infty}\theta_t(\alpha)=\E\lsb f(\Z^{\alpha,x}(\infty))\rsb\quad\text{and}\quad\lim_{t\to\infty}\theta_t'(\alpha)=\E\lsb f'(\Z^{\alpha,x}(\infty))\phib^{\alpha,\xi}(\infty)\rsb.
\end{align}

\begin{lem}
	\label{lem:barh}
	There exists a locally integrable function $\bar\theta:U\to[0,\infty)$ such that $|\theta_t'(\alpha)|\le\bar\theta(\alpha)$ for all $t\ge0$ and $\alpha\in U$.
\end{lem}

\begin{proof}
	Define $\bar{\theta}:U\mapsto[0,\infty]$ by
		$$\bar{\theta}(\alpha):=\sup_{t\ge0}|\theta_t'(\alpha)|,\qquad\alpha\in U.$$
	Let $U_0$ be a compact subset of $U$.
	By Corollary \ref{cor:uniform},
		$$\sup_{\alpha\in U_0}\bar{\theta}(\alpha)\le\norm{f'}_\infty\sup_{\alpha\in U_0}\sup_{s\ge0}\E\lsb\abs{\phib^{\alpha,\xi}(s)}\rsb ds<\infty,$$
	where $\norm{f'}_\infty:=\sup_{x\in G}|f'(x)|<\infty$ since $f\in C_b^1(G)$.
	This proves that $\bar\theta$ is locally bounded, and hence, locally integrable.
\end{proof}

Theorem \ref{thm:sensitivity} is now a simple consequence of this lemma.  

\begin{proof}[Proof of Theorem \ref{thm:sensitivity}]
	Let $-\infty<\alpha_1<\alpha_2<\infty$ be such that $[\alpha_1,\alpha_2]\subset U$.
	By the Fundamental Theorem of Calculus,
		$$\theta_t(\alpha_2)=\theta_t(\alpha_1)+\int_{\alpha_1}^{\alpha_2}\theta_t'(\alpha)d\alpha.$$	
	Letting $t\to\infty$ in the last display and using \eqref{eq:stable}, along with Lemma \ref{lem:barh} and the Lebesgue Dominated Convergence Theorem to interchange the limit and the integral, we obtain
	\begin{align*}
		F(\alpha_2)&=F(\alpha_1)+\lim_{t\to\infty}\int_{\alpha_1}^{\alpha_2}\theta_t'(\alpha) d\alpha\\
		&=F(\alpha_1)+\int_{\alpha_1}^{\alpha_2}\E\lsb f'(\Z^\alpha(\infty))\phib^\alpha(\infty)\rsb d\alpha.
	\end{align*}
	In particular, this implies that for almost every\ $\alpha\in[\alpha_1,\alpha_2]$, $F(\cdot)$ is differentiable at $\alpha$ and its derivative satisfies \eqref{eq:sensitivity}.
	Since $[\alpha_1,\alpha_2]\subset U$ was arbitrary, this completes the proof.
\end{proof}

\appendix

\section{Proof of Lemma \ref{lem-setB}}

Recall that ${\bf 0}$ (resp.\;{\bf 1}) denotes the vector in $\R^J$ with 0 (resp.\ 1) in each component.
In the following proof vector inequalities are interpreted component-wise.

\begin{proof}[Proof of Lemma \ref{lem-setB}]
Due to the fact that $\varrho(Q(\alpha))<1$ and the Perron-Frobenius theorem, it is readily seen that $N^TR(\alpha)=E_J-Q(\alpha)$ is invertible.
Since $N$ is also invertible by assumption, Assumption \ref{ass:independent} holds.

Next, for each $\alpha \in U$, by \cite[Theorem 1 and condition $N_{40}$]{Ple77} and the fact that the diagonal elements of
$N^T R(\alpha)$ are identically $1$, there exists a diagonal matrix $D(\alpha)$, with strictly positive diagonal elements, such that 
$(D(\alpha))^{-1}Q (\alpha) D(\alpha)$ is strictly substochastic, that is, 
	\be\label{eq:malpha}m(\alpha):=(D(\alpha))^{-1}Q (\alpha) D(\alpha){\bf 1}<{\bf 1}.\ee
From the definition of $Q(\alpha)$, this is easily seen to be  equivalent to the statement that there exists a vector $v(\alpha)\in\R_+^J$ with positive elements such that 
	$$N^TR(\alpha)v(\alpha)>{\bf 0}.$$
[The equivalence can be seen by taking $D(\alpha)$ to be the diagonal matrix with diagonal elements equal to the components of $v(\alpha)$ so that $D(\alpha){\bf 1}=v(\alpha)$.]
Moreover, since the map $\alpha\mapsto R(\alpha)$ is continuous, we can choose the map $\alpha\mapsto v(\alpha)$, and thus the maps $\alpha\mapsto D(\alpha)$ and $\alpha\mapsto m(\alpha)$, to be continuous.
It follows from \eqref{eq:malpha} and \cite[Section 2.4]{DupRam99a} that
 	$$B^\alpha:=R(\alpha)H^\alpha=\lcb R(\alpha)x:x\in H^\alpha\rcb,$$
where $H^\alpha:=\{y\in\R^J:|y^i|\le m^i(\alpha)\;\forall\;i\in\allN\}$, is a convex, compact, symmetric set with $0\in(B^\alpha)^\circ$ and satisfies \eqref{eq:setB}.
The continuity of the maps $\alpha \mapsto R(\alpha)$ and $\alpha \mapsto m(\alpha)$ ensures that $\alpha\mapsto B^\alpha$ is continuous in the Hausdorff metric, so Assumption \ref{ass:setB} is satisfied.

We now turn to the  verification of  Assumption \ref{ass:projection}.
Given $x\in\R^J$, the condition that $\pi^\alpha(x) \in G$ is equivalent to the condition that $z:=N^T \pi^\alpha(x)\in\R_+^J$, and, since $N^TR(\alpha)$ and $N$ are non-singular, 
  the condition that $\pi^\alpha(x) -  x \in d(\pi^\alpha(x))$ is equivalent to saying that
  $\pi^\alpha(x) -x = R(\alpha) w$ for some $w\in\R_+^J$ that satisfies 
  $\langle z, w \rangle  = 0$.   
  In other words, $(w,z)$ is a solution to the linear complementarity
  problem associated with the matrix $N^T R(\alpha)$ and input $N^T x \in \mathbb{R}^J$.
  It is well known that a sufficient condition  for this is that $N^T R(\alpha)$ be
  an $\mathcal{M}$-matrix (see, e.g., \cite[Corollary 4]{cotvei72}), which is precisely the condition
  specified in the lemma.   This concludes the proof of the lemma.
\end{proof}

\end{document}